\documentclass[notitlepage,11pt,reqno]{amsart}
\usepackage{amsmath, amssymb, amsfonts, amstext, amsthm, color,tocvsec2}
\usepackage[breaklinks=true]{hyperref}
\usepackage[mathscr]{euscript}
\usepackage{enumerate}
\usepackage[letterpaper]{geometry}
\usepackage{float}
\usepackage{pgf,tikz,ulem}

\newtheorem{thm}{Theorem}[section]
\newtheorem{defn}[thm]{Definition}
\newtheorem{lem}[thm]{Lemma}
\newtheorem{cor}[thm]{Corollary}
\newtheorem{prop}[thm]{Proposition}
\newtheorem{ex}[thm]{Example}
\newtheorem{rem}[thm]{Remark}

\DeclareMathOperator{\qec}{QEC}
\DeclareMathOperator{\diag}{diag}
\DeclareMathOperator{\rk}{rank}
\newcommand{\dg}{\ensuremath{D_G}}
\def\be{\begin{equation}}
\def\ee{\end{equation}}
\def\bea{\begin{eqnarray}}
\def\eea{\end{eqnarray}}
\numberwithin{equation}{section}
\begin{document}
	
	\title[Quadratic Embedding Constants of Cartesian Products and Joins of Graphs]{Quadratic Embedding Constants of Cartesian Products and Joins of Graphs}

	\author{Projesh Nath Choudhury}
	\address[P.N.~Choudhury]{Department of Mathematics, Indian Institute of Technology Gandhinagar,
		Palaj, Gandhinagar 382355, India}
	\email{\tt projeshnc@iitgn.ac.in}

	\author{Raju Nandi}
	\address[R.~Nandi]{Department of Sciences, Indian Institute of Information Technology, Design and Manufacturing, Kurnool, Andhra Pradesh 518008, India}
	\email{\tt rajunandirkm@iiitk.ac.in, rajunandirkm@gmail.com}
	
	\date{\today}
	
	\begin{abstract}
		The quadratic embedding constant (QEC) of a finite, simple, connected graph originated from the classical work of Schoenberg  [\textit{Ann.\ of\ Math.}, 1935] and [\textit{Trans.\ Amer.\ Math.\ Soc.}, 1938] on Euclidean distance geometry. In this article, we study the QEC of graphs in terms of two graph operations: the Cartesian product and the join of graphs. We derive a general formula for the QEC of the join of an arbitrary graph with a regular graph and with a complete multipartite graph. As an application of these results, we explicitly compute the QEC for several classes of graphs and provide new examples of graphs of QE class. We also establish a lower bound for the quadratic embedding constant of the Cartesian product of two arbitrary connected graphs. Furthermore, as an extremal case, we derive concise formulas for the quadratic embedding constants of the Cartesian product of an arbitrary graph G with a complete graph and with a complete bipartite graph, expressed in terms of $\qec(G)$.
	\end{abstract}
	
	

 \maketitle
 
\section{Introduction}
\textit{All graphs in this paper are assumed to be finite, simple, and unweighted. Given a graph $G=(V,E)$, we define $C(V):$ the space of all $\mathbb{R}$-valued functions on $V$, $e:$ the constant function in $C(V)$ taking value $1$, $J :$ matrix with all entries one and $\langle \cdot,\cdot \rangle :$ the canonical inner product on $C(V)\cong \mathbb{R}^{|V|}$, $A_G$: the adjacency matrix of $G$ whose $(v,w)$ entry is 1 if $v$ is adjacent to $w$, and 0 otherwise, for all $v,w \in V$. If $G$ is connected, we define $D_G$: the distance matrix of $G$ with the $(v,w)$ entry given by the length of a shortest path connecting $v \neq w \in V$, and $(D_G)_{vv} = 0\ \forall v \in V$. For a positive integer $n$, we define $[n]:=\{1,\ldots,n\}$. Finally, for a square matrix $M$, $\sigma (M)$ denotes the set of all eigenvalues of $M$ and $\sigma_0(M):=\{\lambda \in \sigma(M): \lambda \hbox{ has an associated eigenvector } x \hbox{ with } \langle  e,x\rangle =0\}$.}
\medskip

The goal of this article is to study the quadratic embedding constant of a connected graph  $G$ using the graph operations, the Cartesian product and the join of graphs. A connected graph $G=(V,E)$ is quadratically embeddable into a Euclidean space $\mathcal{H}$ if there exists a map $\psi:V \to \mathcal{H}$ such that $d(i,j)=\parallel \psi (i) - \psi (j) \parallel ^2$ for all $i,j\in V$, where $\parallel \cdot \parallel $ denotes the norm of $\mathcal{H}$. Such a map $\psi$ from $V$ into $\mathcal{H}$  is called a {\it quadratic embedding} of $G$ and if $G$ admits a quadratic embedding, we say that $G$ is of \textit{QE class}. Graphs of QE class have numerous applications in several branches of mathematics, including quantum probability and non-commutative harmonic analysis \cite{mbozejko, bozejko, graham, haagerup,hora,nobata,obata}.  In 1935, Schoenberg \cite{schoenberg} gave a striking characterization of the QE class in terms of the conditional negative definiteness.  A symmetric matrix $A\in \mathbb{R}^{n\times n}$ is said to be conditionally negative definite if $f^TAf \leq 0$  for all\  $f\in \mathbb{R}^n$ with $\langle e,f \rangle = 0$. Schoenberg showed that a connected graph $G$ is of QE class if and only if $D_G$ is conditionally negative definite. Inspired by this result, Obata--Zakiyyah \cite{zakiyyah} introduced the notion of {\it quadratic embedding constant} of $G$, denoted and defined as follows.
\begin{center}
	$\qec(G):= \max\{\langle f,\dg f \rangle ;\  f\in C(V),\  \langle f,f \rangle = 1,\  \langle e,f \rangle = 0 \}.$
\end{center}

In particular, Schoenberg's result states that a connected graph $G$ belongs to the QE class if and only if $\qec (G)\leq 0$. Since then, the QEC of graphs has been studied by several authors \cite{baskoro,lou,wmlotkowski,mo,bata2,bata,zakiyyah}. For instance, the quadratic embedding constants of some well-known classes of graphs, including the complete graph, cycle, path, and complete multipartite graph are studied in \cite{wmlotkowski,bata,zakiyyah}.  In \cite{pnc}, we provided  tight lower and upper bounds for the QEC of a tree, and showed that it lies between a path and a star for all trees.

Determining the quadratic embedding constant for a more general class of graphs is a challenging problem. A well-studied question involving the QEC asks to study the QEC of a connected graph $G$ by factorizing it into smaller graphs using a graph operation.  In the literature, QEC of graphs is explored using several graph operations, including Cartesian product \cite{zakiyyah}, star product \cite{baskoro,mo}, lexicographic product \cite{lou}, graph joins \cite{lou}, and cluster of graphs \cite{pnc}. In \cite{nbata}, Obata provided the $\qec$ formula for the wheel graph $K_1+C_n$. In \cite{lou}, Lou--Obata--Huang extended this result by giving a formula for $\qec (G_1+G_2)$, when $G_1$ and $G_2$ are regular graphs. Recently, Młotkowski--Obata \cite{wm} gave a formula for the quadratic embedding constant of the join of  $\overline{K}_m$ and an arbitrary graph $G$. Our first main result extends their work by replacing $\overline{K}_m$ with an arbitrary regular graph. We provide a formula for $\qec (G_1+G_2)$, when $G_1$ is a regular graph and $G_2$ is an arbitrary graph. Note that if $G_1+G_2$ is complete then $\qec(G_1+G_2)=-1$.
\begin{thm}\label{theorem1}
	Let $G_1$ and $G_2$ be two disjoint graphs on $m\geq 1$ and $n\geq 1$ vertices, respectively, such that $G_1+G_2$ is not complete. Suppose $G_1$ is an $r$-regular graph. Then \[\qec (G_1+G_2)=-2-\min\left(\bigcup\limits_{i=1}^4(\Lambda_i (G_1+G_2)\cap (-\infty , -1))\right),\] where
	\begin{eqnarray*}
		\Lambda_1 (G_1+G_2)&:=&
			\{r-m\} \cap (\sigma(A_{G_1})\cup \sigma(A_{G_2}-J)), \\
		\Lambda_2 (G_1+G_2)&:=& \{r-2m\} \cap \sigma(A_{G_2}), \\
		\Lambda_3 (G_1+G_2)&:=&\left(\sigma_0(A_{G_1})\cup \sigma_0(A_{G_2})\right)\backslash \{r-m, r-2m\},\\
		\Lambda_4 (G_1+G_2)&:=&\{\lambda \in \mathbb{R}\setminus(\{r-m, r-2m\}\cup \sigma(A_{G_1})\cup \sigma(A_{G_2})): m=(\lambda +2m-r)\langle e,(A_{G_2}-\lambda I)^{-1}e \rangle\}.
	\end{eqnarray*}
\end{thm}

As an application of our first main result, we explicitly compute the quadratic embedding constant for several well known classes of graphs and show that each of the above four sets is pivotal in computing $\qec(G)$. 

In the literature, the quadratic embedding constant of a graph join $G_1+G_2$ is studied only when either $G_1$ or $G_2$ is regular \cite{lou,wm,nbata}. So, it is natural to ask about the case in which neither of the graphs is regular. Our next main result addresses this question. In particular, we obtain the quadratic embedding constant of the join of a complete multipartite graph and an arbitrary graph $G$. To state our next main result, we need to define a real-valued function.

For  distinct positive integers $m_{i_1},m_{i_2},\ldots ,m_{i_q}$ and any positive integers $a_1,a_2,\ldots ,a_q$, define the function  $P:\mathbb{R}\setminus\{-m_{i_1},-m_{i_2},\ldots ,-m_{i_q}\}\to \mathbb{R}$ via
\begin{equation}\label{eqplambda}
	P(\lambda):= 1+\sum _{p=1}^q\frac{a_pm_{i_p}}{\lambda +m_{i_p}}.
\end{equation}
Notice that all the zeros of $P$ are negative real numbers.  To see this, it suffices to show that the polynomial $f(\lambda)=  \prod\limits _{p=1}^q(\lambda +m_{i_p})+\sum\limits _{p=1}^qa_pm_{i_p}\prod \limits_{\substack{s=1 \\ s\neq p}}^q (\lambda +m_{i_s})$ has $q$ negative real zeros. Without loss of generality assume that $m_{i_1}>m_{i_2}>\cdots >m_{i_q}$. By a simple calculation one can verify that $f(-m_{i_j})f(-m_{i_{j+1}})<0$  for all $j\in [q-1]$. Thus each interval $(-m_{i_j},-m_{i_{j+1}}),1\leq j\leq q-1$ contains a root of $f$. Since $f$ is a real polynomial with positive coefficients, all $q$ roots of $f$ are negative real numbers. 
\begin{thm}\label{theorem2}
Let $G$ be a graph with $n\geq 1$ vertices and $K_{m_1,m_2,\ldots ,m_k}$ be a complete multipartite graph such that $K_{m_1,m_2,\ldots ,m_k}+G$ is not complete. Suppose $m_{i_1},m_{i_2},\ldots ,m_{i_q}$ are distinct with $m_{i_p}$ appears $a_p$ times for all $p\in[q]$. Let $P$ be the function defined in \eqref{eqplambda} and $\lambda _1,\lambda _2,\ldots , \lambda _q$ be the zeros of $P$. 
Then \[\qec (K_{m_1,\ldots ,m_k}+G)=-2-\min \left(\bigcup\limits_{i=1}^4(\Lambda_i (K_{m_1,\ldots ,m_k}+G)\cap (-\infty , -1))\cup \{-m_{i_p}: p\in[q], a_p\geq 2, m_{i_p}\neq 1\}\right),\] where 
	\begin{eqnarray*}
		\Lambda_1 (K_{m_1,\ldots ,m_k}+G)&:=& \{-m_{i_p}: p\in[q], a_p=1\}\cap\sigma(A_G-J),\\
		\Lambda_2 (K_{m_1,\ldots ,m_k}+G)&:=&\{\lambda _p: p\in[q]\}\cap \sigma(A_G),\\
			\Lambda_3 (K_{m_1,\ldots ,m_k}+G)&:=&\sigma_0(A_G)\backslash \{-m_{i_p},\lambda _p: p\in[q]\},\\
			\Lambda_4 (K_{m_1,\ldots ,m_k}+G)&:=&\{\lambda \in \mathbb{R}\setminus \left(\{-m_{i_p},\lambda _p: p\in[q]\}\cup  \sigma(A_G)\right):  P(\lambda)\langle e,(A_G-\lambda I)^{-1}e \rangle=P(\lambda)-1\}.
		\end{eqnarray*}
\end{thm}
If $q=1$ in the above theorem i.e, $m_1=m_2=\cdots =m_k=m$, then $K_{m,m,\ldots ,m}$ is a regular graph of degree $(k-1)m$ and $-m$ is an eigenvalue of  $A_{K_{m,m,\ldots ,m}}$. A straightforward calculation shows that Theorem \ref{theorem2} can be derived from Theorem \ref{theorem1}.

We next turn our attention to the Cartesian product of graphs. We give a lower bound of $\qec (G_1\times G_2)$ in terms of $\qec (G_1)$ and $\qec (G_2)$.
\begin{prop}\label{qecrelation}
	Let $G_1$ and $G_2$ be two connected graphs with $m\geq 1$ and $n\geq 1$ vertices, respectively. Then
	\begin{equation}\label{extremalinequality}
		\max \{n\qec (G_1),m \qec (G_2)\}\leq\qec (G_1\times G_2).
	\end{equation}
\end{prop}
We next study the extremal case of the inequality \eqref{extremalinequality}. In \cite{zakiyyah}, Obata--Zakiyyah showed that the quadratic embedding constant of the Cartesian product of two graphs of QE class is zero. It is natural to study the case when at least one of the graphs in the Cartesian product does not belong to the QE class. Using Proposition \ref{qecrelation}, we show that the Cartesian product of two connected graphs $G_1$ and $G_2$ is of QE class if and only if both $G_1$ and $G_2$ are of QE class (see Corollary \ref{qeccatprodiff}). Thus, if a graph $G$ is of QE class, then $K_m\times G$ is of QE class and $\qec(K_m\times G)=0$. 
In our next main result, for a connected graph $G$ of non-QE class, we obtain $\qec(K_m\times G)$ in terms of $\qec(G)$.
\begin{thm}\label{theorem4}
	Let $G$ be a connected graph of non-QE class with $n$ vertices and $K_m$ be a complete graph. Then $\qec (K_m\times G)=m\qec (G)$.
	\end{thm}
	
In our final main result,  we derive $\qec(K_{m,n}\times G)$ in terms of $\qec(G)$ when either $G$ or $K_{m,n}$ is not of QE class.	
	\begin{thm}\label{theorem5}
Let $G$ be a connected graph with $l$ vertices and $K_{m,n}$ be a complete bipartite graph.
\begin{itemize}
	\item [(i)] If $K_{m,n}$ is of QE class and $G$ is of non-QE class, then $\qec (K_{m,n}\times G)=(m+n)\qec (G)$.
	\item [(ii)] If $K_{m,n}$ is of non-QE class, then $\qec (K_{m,n}\times G)=\max \{(m+n)\qec (G),l\frac{n(m-2)+m(n-2)}{m+n}\}$.
\end{itemize}
\end{thm}

As an immediate consequence of Theorem \ref{theorem5}, one can conclude that for a connected graph $G$ of QE class, the quadratic embedding constant of $K_{m,n}\times G$ does not depend on the structure of $G$.

\begin{cor}\label{cor1}
	Let $m\geq 3,n\geq 2$ and $G$ be a connected graph of QE class on $l$ vertices . Then $\qec (K_{m,n}\times G)=l\frac{n(m-2)+m(n-2)}{m+n}$. In particular,  $\qec (K_{m,n}\times T)=l\frac{n(m-2)+m(n-2)}{m+n}$ for any tree $T$ on $l$ vertices.
\end{cor}
\begin{rem}
By a classical result of Graham--Pollak \cite{Graham-Pollak}, it is well known that the determinant of the distance matrix of tree is independent of the tree structure.  Corollary \ref{cor1} may be regarded being similar in spirit.
\end{rem}

\noindent\textbf{Organization of the paper:} The remaining sections are devoted to proving our main results above. In Section \ref{graphjoinsection}, we recap the definition of graph join, some preliminary results, and prove our first two main results, Theorems \ref{theorem1} and \ref{theorem2}. In the final section, we first recall the Cartesian product of graphs and the quadratic embedding constant of a complete bipartite graph, and then prove Proposition \ref{qecrelation}, and Theorems \ref{theorem4} and \ref{theorem5}. As an application of the main results, we derive the quadratic embedding constant for various classes of graphs and provide new examples of graphs of QE class.

\section{Quadratic embedding constant and Graph Join}\label{graphjoinsection}
We begin by proving Theorems \ref{theorem1} and \ref{theorem2} -- quadratic embedding constant of the join of two graphs with at least one of them is not regular. The proof requires some preliminary results. The first result gives a sharp lower bound for the quadratic embedding constant of a connected graph.
\begin{prop}\label{qeccomplete}\cite{lou}
	Let $G$ be a connected graph on $n$ vertices. Then $\qec(G)\geq -1$ and equality holds if and only if $G$ is a complete graph $K_n$.
\end{prop}

The following result provides a handy technique for computing the quadratic embedding constant of a connected graph using the Lagrange multipliers.
\begin{prop}\cite[Section 3]{nbata} \label{qecformula}
	Let $M\in \mathbb{R}^{l\times l}$ ($l\geq 3$) be a symmetric matrix and $\mathcal{S}(M)$ be the set of all stationary points $(f,\lambda ,\mu)\in \mathbb{R}^l\times \mathbb{R} \times \mathbb{R}$ of
	\begin{center}
		$\psi (f,\lambda ,\mu)=\langle f,Mf \rangle -\lambda (\langle f,f \rangle -1)-\mu \langle e,f \rangle ,$
	\end{center}
	or equivalently, $(f,\lambda ,\mu)\in \mathbb{R}^l\times \mathbb{R} \times \mathbb{R}$ satisfies the system of the following three equations
	\begin{equation}\label{equivalentequation}
		(M-\lambda I)f=\frac{\mu}{2}\ e,\ \ \ \  \langle f,f \rangle =1\ \  and\ \  \langle e,f \rangle =0.
	\end{equation}
	Then $\mathcal{S}(M)$ is nonempty. Moreover,
	\begin{itemize}
		\item [(i)] $\max\{\langle f,M f \rangle ;\  f\in \mathbb{R}^l,\  \langle f,f \rangle = 1,\  \langle e,f \rangle = 0 \} = \max \{\lambda : (f,\lambda ,\mu)\in \mathcal{S}(M) \}$.
		\item [(ii)] 	$\min\{\langle f,M f \rangle ;\  f\in \mathbb{R}^l,\  \langle f,f \rangle = 1,\  \langle e,f \rangle = 0 \} = \min \{\lambda : (f,\lambda ,\mu)\in \mathcal{S}(M) \}$.
	\end{itemize}
\end{prop}
To prove Theorems \ref{theorem1} and \ref{theorem2}, we now recall the definition of the join of two disjoint graphs. Note that two graphs $G_1=(V_1, E_1)$ and $G_2=(V_2, E_2)$ is disjoint if $V_1\cap V_2=\emptyset$.
\begin{defn}\label{gjoin}
Let $G_1=(V_1, E_1)$ and $G_2=(V_2, E_2)$ be two disjoint graphs. Then the join of $G_1$ and $G_2$, denoted by $G_1+G_2$, is a graph with vertex set $V=V_1\cup V_2$ and edge set $E=E_1\cup E_2\cup \{(u_i,v_j);\ u_i\in V_1,\  v_j\in V_2\}$.
\end{defn}
\begin{ex}
Example of a join graph $P_4+C_3$:
\begin{figure}[H]
\begin{center}
\begin{tikzpicture}[scale=1]
\draw  (1,0)-- (2,0);
\draw  (2,0)-- (3,0);
\draw  (3,0)-- (4,0);
\begin{scriptsize}
\fill (1,0) circle (2.5pt);
\draw (1.0,-0.4) node {$u_1$};
\fill (2,0) circle (2.5pt);
\draw (2.,-0.4) node {$u_2$};
\fill (3,0) circle (2.5pt);
\draw (3.,-0.4) node {$u_3$};
\fill (4,0) circle (2.5pt);
\draw (4.,-0.4) node {$u_4$};
\end{scriptsize}
\end{tikzpicture}
\hspace{1cm}
\begin{tikzpicture}[scale=1]
\draw  (1,1)-- (2,2);
\draw  (2,2)-- (3,1);
\draw  (3,1)-- (1,1);
\begin{scriptsize}
\fill (2.,2.) circle (2.5pt);
\draw (2.,2.35) node {$v_1$};
\fill (1.,1.) circle (2.5pt);
\draw (1.,0.65) node {$v_2$};
\fill (3.,1.) circle (2.5pt);
\draw (3.,0.65) node {$v_3$};
\end{scriptsize}
\end{tikzpicture}
\hspace{2cm}
\begin{tikzpicture}[scale=1]
\draw  (1,3)-- (1,2);
\draw  (1,2)-- (1,1);
\draw  (1,1)-- (1,0);
\draw  (2.9923874167026288,2.604520628645632)-- (4.4038271330841035,1.5230278589507376);
\draw  (4.4038271330841035,1.5230278589507376)-- (2.9923874167026288,0.4048743174017789);
\draw  (2.9923874167026288,0.4048743174017789)-- (2.9923874167026288,2.604520628645632);
\draw  (1,3)-- (2.9923874167026288,2.604520628645632);
\draw  (1,3)-- (4.4038271330841035,1.5230278589507376);
\draw  (1,3)-- (2.9923874167026288,0.4048743174017789);
\draw  (1,2)-- (2.9923874167026288,2.604520628645632);
\draw  (1,2)-- (4.4038271330841035,1.5230278589507376);
\draw  (1,2)-- (2.9923874167026288,0.4048743174017789);
\draw  (1,1)-- (2.9923874167026288,2.604520628645632);
\draw  (1,1)-- (4.4038271330841035,1.5230278589507376);
\draw  (1,1)-- (2.9923874167026288,0.4048743174017789);
\draw  (1,0)-- (2.9923874167026288,2.604520628645632);
\draw  (1,0)-- (4.4038271330841035,1.5230278589507376);
\draw  (1,0)-- (2.9923874167026288,0.4048743174017789);
\begin{scriptsize}
\fill (1,3) circle (2.5pt);
\draw (0.65,3.) node {$u_1$};
\fill (1,2) circle (2.5pt);
\draw (0.65,2.) node {$u_2$};
\fill (1,1) circle (2.5pt);
\draw (0.65,1.) node {$u_3$};
\fill (1,0) circle (2.5pt);
\draw (0.65,0.) node {$u_4$};
\fill (2.9923874167026288,2.604520628645632) circle (2.5pt);
\draw (3.0,2.9) node {$v_2$};
\fill (2.9923874167026288,0.4048743174017789) circle (2.5pt);
\draw (3.0,0.1) node {$v_3$};
\fill (4.4038271330841035,1.5230278589507376) circle (2.5pt);
\draw (4.7,1.55) node {$v_1$};
\end{scriptsize}
\end{tikzpicture}
\end{center}
\caption{$P_4$, $C_3$ and $P_4+C_3$ (Left to right)}
\label{figure1}
\end{figure}
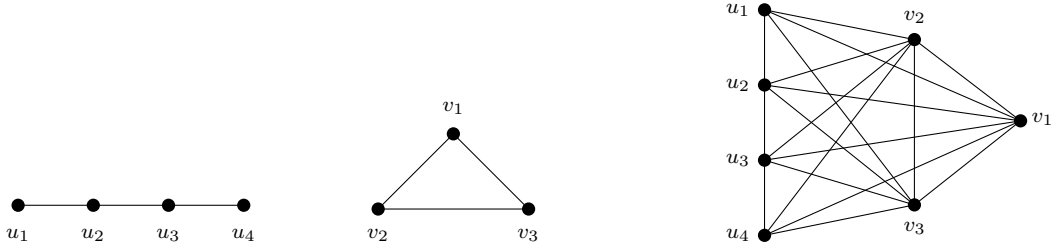
\end{ex}
In 2017, Obata \cite{nbata} derived an interesting formula for the quadratic embedding constant of a graph join using its adjacency matrix.
\begin{prop}\cite[Proposition 2.1.]{nbata}\label{graphjoinqec} Let $G_1=(V_1, E_1)$ and $G_2=(V_2, E_2)$ be two disjoint graphs and let $V=V_1\cup V_2$. Then
\[\qec (G_1+G_2)=-2-\min\{\langle f,A_{G_1+G_2} f \rangle:\  f\in C(V),\  \langle f,f \rangle = 1,\  \langle e,f \rangle = 0 \}.\]
\end{prop}
Let $G_1=(V_1, E_1)$ and $G_2=(V_2, E_2)$ be two disjoint graphs with $|V_1|=m\geq 1$ and $|V_2|=n\geq 1$. Since a connected graph on two vertices is the complete graph $K_2$, we assume $m+n\geq 3$. Note that a graph join $G_1+G_2$ is a connected graph, and its diameter is either one or two. Thus one can write the adjacency matrix of $G_1+G_2$  as the block matrix $A_{G_1+G_2}=\begin{pmatrix}
	A_{G_1} & J\\
	J & A_{G_2}
\end{pmatrix}$. For the adjacency matrix $A_{G_1+G_2}$, the equation \eqref{equivalentequation} reduces to the following system of equations:
\begin{eqnarray}
	(A_{G_1}-J-\lambda I)x&=&\frac{\mu}{2}\  e, \label{firstequation} \\
	(A_{G_2}-J-\lambda I)y&=&\frac{\mu}{2}\  e, \label{secondequation} \\
	\langle e,x \rangle + \langle e,y \rangle& =&0, \label{thirdequation} \\
    \langle x,x \rangle + \langle y,y \rangle& =&1, \label{fourthequation}
\end{eqnarray}
where $x\in C(V_1) \cong \mathbb{R}^m$ and $y\in C(V_2) \cong \mathbb{R}^n$. Let $\mathbb{S}(A_{G_1+G_2})$ be the set of all solutions $(x,y,\lambda, \mu )\in \mathbb{R}^m \times \mathbb{R}^n \times \mathbb{R} \times \mathbb{R}$ of the system of equations \eqref{firstequation}--\eqref{fourthequation}. Then all the stationary points $\mathcal{S}(A_{G_1+G_2})$ of the function
\begin{equation*}
		\psi (x,y,\lambda ,\mu)=\Big \langle \begin{pmatrix}
x\\
y
\end{pmatrix},A_{G_1+G_2}\begin{pmatrix}
x\\
y
\end{pmatrix} \Big \rangle -\lambda (\langle x,x \rangle +\langle y,y \rangle-1)-\mu (\langle e,x \rangle + \langle e,y \rangle )
\end{equation*}
coincides with the points in $\mathbb{S}(A_{G_1+G_2})$. By Propositions \ref{graphjoinqec} and \ref{qecformula}, one can conclude that
\begin{equation}\label{graphjoinqec2}
\qec (G_1+G_2)=-2-\min \{\lambda : (x,y,\lambda, \mu )\in \mathbb{S}(A_{G_1+G_2})\}.
\end{equation}
Define $\Lambda (G_1+G_2) :=\{\lambda : (x,y,\lambda, \mu )\in \mathbb{S}(A_{G_1+G_2})\}$. Then, we can rewrite \eqref{graphjoinqec2} as
\begin{equation}\label{newgraphjoinqec}
\qec (G_1+G_2)=-2-\min \Lambda (G_1+G_2).
\end{equation}

With these basic ingredients, we now prove Theorem \ref{theorem1}. 
\begin{proof}[Proof of Theorem \ref{theorem1}]
	
 Note that a graph join $G_1+G_2$ is complete if and only if $G_1$ and $G_2$ both are complete graphs. 
 Since $G_1+G_2$ is not complete, by Proposition \ref{qeccomplete}, $\qec(G_1+G_2)>-1$ and using \eqref{graphjoinqec2}, it is sufficient to consider all $\lambda <-1$ such that $(x,y,\lambda, \mu )\in \mathbb{S}(A_{G_1+G_2})$. By \eqref{newgraphjoinqec}, $\qec(G_1+G_2)=-2-\min \left(\Lambda (G_1+G_2)\cap (-\infty , -1)\right).$ Thus, to complete the proof,  it suffices to show that
 \[\Lambda (G_1+G_2)\cap (-\infty , -1)=\bigcup\limits_{i=1}^4(\Lambda_i (G_1+G_2)\cap (-\infty , -1)).\] 
 We now split the proof into several lemmas.

\begin{lem}\label{lemma1}
Let $\lambda = r-m$. Then $\lambda$ appears in the solution of  \eqref{firstequation}-\eqref{fourthequation} if and only if $\lambda \in \sigma(A_{G_1})\cup \sigma(A_{G_2}-J)$.
\end{lem}
\begin{proof}
Suppose $(x,y,r-m,\mu)$ is a solution of \eqref{firstequation}-\eqref{fourthequation}. Since $G_1$ is a $r$-regular graph, $A_{G_1}e=re$ and the equation \eqref{firstequation} implies $\mu =0$. From equations \eqref{firstequation} and \eqref{secondequation} we have
\begin{eqnarray}
	\left(A_{G_1}-(r-m) I\right)x&=&\langle e,x \rangle e, \nonumber\\ 
	(A_{G_2}-J-(r-m) I)y&=&0. \nonumber 
\end{eqnarray}
Notice that, by \eqref{fourthequation}, either $x$ or $y$ is nonzero. Thus, if $\langle e,x \rangle =0$, then from the above equations, one can conclude that $r-m \in \sigma(A_{G_1})\cup \sigma(A_{G_2}-J)$. Let $\langle e,x \rangle \neq 0$. Then $\langle e,y \rangle \neq 0$ from \eqref{thirdequation}. Thus $y\neq 0$ and $r-m\in \sigma(A_{G_2}-J)$.

To prove the converse, suppose $r-m \in \sigma(A_{G_1})\cup \sigma(A_{G_2}-J)$. We now consider two cases. 

\noindent {\bf Case I.} $r-m \in \sigma(A_{G_2}-J)$: Then there exists $0\neq y_0\in \mathbb{R}^n$ such that
\begin{equation*}
(A_{G_2}-J+(m-r) I)y_0=0.
\end{equation*}
Since $G_1$ is a $r$-regular graph on $m$ vertices, $r-m$ is an eigenvalue of $A_{G_1}-J$ and the corresponding eigenvector is $\alpha e$ for any $\alpha \in \mathbb{R}$. Thus $(x=\alpha e,y=\beta y_0,\lambda = r-m, \mu =0)$ is a solution of \eqref{firstequation}-\eqref{fourthequation}, where  $\alpha,\beta \in \mathbb{R}$ are chosen appropriately from equations \eqref{thirdequation}-\eqref{fourthequation} in terms of $\langle e,y_0 \rangle $ and $\langle y_0,y_0 \rangle $.

\noindent {\bf Case II.} $r-m \in \sigma(A_{G_1})$: Then there exists $0\neq x_0\in \mathbb{R}^m$ such that
\begin{equation*}
(A_{G_1}+(m-r) I)x_0=0.
\end{equation*}
Since $r\neq r-m$ and $e$ is an eigenvector of $A_{G_1}$ corresponding to the eigenvalue $r$, $\langle e,x_0 \rangle =0$. Thus, $(x=\frac{x_0}{\parallel x_0 \parallel _2},y=0,\lambda = r-m, \mu =0)$ satisfies equations \eqref{firstequation}-\eqref{fourthequation}.
\end{proof}
We next consider $\lambda \neq r-m$. To proceed further, we rewrite the equations \eqref{firstequation} and \eqref{secondequation} using relation \eqref{thirdequation}. By a straightforward calculation, we have
\begin{equation}\label{innerproduct}
\langle e,x \rangle = -\langle e,y \rangle = \frac{m}{r-m-\lambda} \frac{\mu}{2},
\end{equation}
\begin{equation}\label{firstequation4}
(A_{G_1}-\lambda I)x=Jx+\frac{\mu}{2}e=\langle e,x \rangle e+\frac{\mu}{2}e=\frac{m}{r-m-\lambda}\frac{\mu}{2}e+\frac{\mu}{2}e=\frac{r-\lambda}{r-m-\lambda}\frac{\mu}{2}e,
\end{equation}
and
\begin{equation}\label{secondequation3}
(A_{G_2}-\lambda I)y=Jy+\frac{\mu}{2}e=\langle e,y \rangle e+\frac{\mu}{2}e=-\frac{m}{r-m-\lambda}\frac{\mu}{2}e+\frac{\mu}{2}e=\frac{r-2m-\lambda}{r-m-\lambda}\frac{\mu}{2}e.
\end{equation}
\begin{lem}\label{lemma2}
Let $\lambda = r-2m$. Then $\lambda$ appears in the solution of  \eqref{firstequation}-\eqref{fourthequation} if and only if $\lambda \in \sigma(A_{G_2})$.
\end{lem}
\begin{proof}
Suppose $(x,y,r-2m,\mu) \in \mathbb{R}^m \times \mathbb{R}^n \times \mathbb{R} \times \mathbb{R}$ is a solutions of \eqref{firstequation}-\eqref{fourthequation}. Then $(x,y,r-2m, \mu)$ satisfies \eqref{innerproduct}-\eqref{secondequation3}. From \eqref{secondequation3}, we get
\[
\left(A_{G_2}+(2m-r) I\right)y=0.
\]
We next show that $y\neq 0$. By Gershgorin's theorem, $r$ is the eigenvalue of $A_{G_1}$ with the largest absolute value. Since $(2m-r)>r$, $\{A_{G_1}+(2m-r)I\}$ is an invertible matrix. If $y=0$, then $\mu =0$ by \eqref{innerproduct} and $x=0$ by \eqref{firstequation4}, which contradicts $\langle x,x \rangle + \langle y,y \rangle=1$. Thus $y \neq 0$ and $r-2m \in \sigma(A_{G_2})$.

Conversely, suppose that $r-2m \in \sigma(A_{G_2})$. Then there exists $0\neq y_0\in \mathbb{R}^n$ such that
\begin{equation*}
(A_{G_2}+(2m-r) I)y_0=0.
\end{equation*}
Let $x=\frac{1}{m}\frac{\mu}{2}e$ and $y=cy_0$, where $\mu, c\in \mathbb{R}$. Then $(x,y,r-2m,\mu)$ satisfies \eqref{firstequation4} and \eqref{secondequation3} for any $\mu, c\in \mathbb{R}$. Now, choose $c$ and $\mu$ such that
\begin{eqnarray*}
	\frac{\mu}{2}+c\langle e,y_0 \rangle&=&0, \\
    \frac{1}{m}\Big (\frac{\mu}{2} \Big )^2 +c^2\langle y_0,y_0 \rangle & =&1.
\end{eqnarray*}
Then, $(x=\frac{1}{m}\frac{\mu}{2}e,y=cy_0,\lambda =r-2m,\mu)$ is a solution of equations \eqref{firstequation}-\eqref{fourthequation}.
\end{proof}

\begin{lem}\label{lemma4}
	Let $\lambda \in (\sigma(A_{G_1})\cup \sigma(A_{G_2}))\backslash \{r-m,r-2m\}$. Then $\lambda$ appears  in the solution of  \eqref{firstequation}-\eqref{fourthequation} if and only if $\lambda \in \sigma_0(A_{G_1})\cup \sigma_0(A_{G_2})$.
\end{lem}
\begin{proof}
	Suppose $\lambda \in \sigma_0(A_{G_1})\cup \sigma_0(A_{G_2})$. If $\lambda \in \sigma_0(A_{G_1})$, then there exists $x_0\neq 0$ such that
	\begin{center}
		$A_{G_1}x_0=\lambda x_0$ and $\langle e,x_0\rangle =0$.
	\end{center}
	Then $(x=\frac{x_0}{\parallel x_0\parallel_2},y=0, \lambda , \mu =0)$ is a solution of \eqref{firstequation}-\eqref{fourthequation}. Similarly, if $\lambda \in \sigma_0(A_{G_2})$, there exists $y_0\neq 0$ such that $(x=0,y=\frac{y_0}{\parallel y_0\parallel_2}, \lambda , \mu =0)$ satisfies \eqref{firstequation}-\eqref{fourthequation}.
	
	To prove the converse, let $(x,y,\lambda, \mu)\in \mathbb{R}^m \times \mathbb{R}^n \times \mathbb{R} \times \mathbb{R}$ be a solution of the system of equations \eqref{firstequation}-\eqref{fourthequation}.  Then either $x$ or $y$ is non-zero. 
	 We now split the remaining part of the proof into two cases.
	
	\noindent {\bf Case I.} $\mu =0$:  By equations \eqref{innerproduct}--\eqref{secondequation3}), we have $\langle e,x \rangle = \langle e,y \rangle=0$ and
\[(A_{G_1}-\lambda I)x=0, ~~(A_{G_2}-\lambda I)y=0\]
	Since both $x$ and $y$ can not be zero at the same time,  $\lambda \in \sigma_0(A_{G_1})\cup \sigma_0(A_{G_2})$. 
	
	\noindent {\bf Case II.} $\mu \neq 0$: Let $\lambda \in \sigma (A_{G_2})$. Since $\lambda\neq r-m, r-2m$, by \eqref{secondequation3}, $y \neq 0$ and it is a solution of  the equation \eqref{secondequation3}.
	This implies $\rk(A_{G_2}-\lambda I,~e)= \rk(A_{G_2}-\lambda I)$. Thus $\rk(A_{G_2}-\lambda I,~e)<n$,  since $A_{G_2} \in \mathbb{R}^{n \times n}$ and $\lambda \in \sigma (A_{G_2})$. Hence there exists a $z \in \mathbb{R}^n$ with $\langle z, z\rangle=1$ such that $z^T(A_{G_2}-\lambda I,~e)=0$.  Since $(A_{G_2}-\lambda I)$ is a symmetric matrix,
	\begin{equation*}
		(A_{G_2}-\lambda I)z=0\ \ \  \mbox{and}\ \ \  \langle e,z\rangle =0.
	\end{equation*}
	Thus $\lambda \in \sigma_0(A_{G_2})$.  
	
	Since $\lambda<-1$, a similar argument shows that if $\lambda \in \sigma (A_{G_1})$, then $\lambda \in \sigma_0(A_{G_1})$. This concludes the proof.
\end{proof}

\begin{lem}\label{lemma3}
Let $\lambda \in \mathbb{R}\backslash ( \{r-m,r-2m\}\cup \sigma(A_{G_1})\cup \sigma(A_{G_2}))$. Then $\lambda$ appears in the solution of  \eqref{firstequation}-\eqref{fourthequation} if and only if
\begin{equation}\label{r3equation}
m-(\lambda +2m-r)\langle e,(A_{G_2}-\lambda I)^{-1}e \rangle=0.
\end{equation}
\end{lem}
\begin{proof}
Suppose $\lambda$ appears in the solutions of \eqref{firstequation}-\eqref{fourthequation}. Then there exist $x\in \mathbb{R}^{m}, y\in \mathbb{R}^{n}$ and $\mu \in \mathbb{R}$ such that $(x,y,\lambda , \mu)$ satisfies \eqref{innerproduct}-\eqref{secondequation3} and \eqref{fourthequation}. Since $\lambda \notin \sigma(A_{G_2})$, $y=\frac{r-2m-\lambda}{r-m-\lambda} \frac{\mu}{2}(A_{G_2}-\lambda I)^{-1}e$ by equation \eqref{secondequation3}.  Thus, from \eqref{thirdequation} and \eqref{innerproduct}  we have
\begin{eqnarray*}
    \frac{m}{r-m-\lambda} \frac{\mu}{2} +\Big \langle e,\frac{r-2m-\lambda}{r-m-\lambda} \frac{\mu}{2}(A_{G_2}-\lambda I)^{-1}e \Big \rangle & =&0\\
    \Rightarrow \frac{\mu}{2}\Big (m - (\lambda +2m-r)\langle e,(A_{G_2}-\lambda I)^{-1}e \rangle \Big )& =&0
\end{eqnarray*}
If $\mu=0$, then $x=y=0$ by \eqref{firstequation4} and \eqref{secondequation3} which contradicts \eqref{fourthequation}. Thus \[m-(\lambda +2m-r)\langle e,(A_{G_2}-\lambda I)^{-1}e \rangle=0.\]

Conversely, suppose $\lambda$ satisfies \eqref{r3equation}. Take
\begin{equation*}
x=\frac{\mu}{2}\frac{1}{r-m-\lambda}e\ \  \mbox{and}\ \  y=\frac{\mu}{2}\frac{r-2m-\lambda}{r-m-\lambda}(A_{G_2}-\lambda I)^{-1}e
\end{equation*}
such that $\mu \in \mathbb{R}$ satisfies the following equation
\begin{equation*}
\Big (\frac{\mu}{2}\Big )^2\left( \Big (\frac{1}{r-m-\lambda} \Big )^2m+\Big (\frac{r-2m-\lambda}{r-m-\lambda} \Big )^2\parallel (A_{G_2}-\lambda I)^{-1}e \parallel_2^2 \right)=1.
\end{equation*}
Then $(x,y,\lambda,\mu)$ fulfills the equations \eqref{firstequation}-\eqref{fourthequation}).
\end{proof}
 Using the above four lemmas, one can verify that
 \[\Lambda(G_1+G_2)\cap (-\infty , -1)=\bigcup\limits _{i=1}^4\big (\Lambda_i (G_1+G_2)\cap (-\infty , -1)\big ).\]This completes the proof.
 \end{proof}
Using our first main result, we recover the QEC of several classes of graphs known in the literature. The first one is derived in  \cite[Theorem 3.1]{lou}
\begin{ex}
Suppose $G_1$ is $r$-regular and $G_2$ is a $\tilde{r}$-regular disjoint graphs on $m\geq 1$ and $n\geq 1$ vertices respectively. If $r-m\neq \tilde{r}-n$ and $r-m\in \sigma (A_{G_2}-J)$, then using the fact $\tilde{r}-n\in \sigma (A_{G_2}-J)$, we can conclude that $r-m\in \sigma (A_{G_2})$. Also, note that $\sigma_0(A_{G_1})=\sigma(A_{G_1})\backslash \{r\}$ and $\sigma_0(A_{G_2})=\sigma(A_{G_2})\backslash \{\tilde{r}\}$. Since $A_{G_2}e=\tilde{r}e$, $(A_{G_2}-\lambda I)^{-1}e=(\tilde{r}-\lambda )^{-1}e$ and the equation
\[
    m-(\lambda +2m-r)\langle e,(A_{G_2}-\lambda I)^{-1}e \rangle=0
\]
has a unique solution $\lambda =\frac{m\tilde{r}+nr-2mn}{m+n}<-1$. Note that $\lambda =r-m$ if and only if  $r-m=\tilde{r}-n$. Thus
\[
	\Lambda_1 (G_1+G_2) \begin{cases}
		=\{r-m\}, & \mbox{if }r-m= \tilde{r}-n, \\
		\subseteq (\sigma (A_{G_1})\cup \sigma (A_{G_2})), & \mbox{otherwise }.
	\end{cases}
\]
\[
\Lambda_2 (G_1+G_2) \subseteq \sigma (A_{G_2}),\quad \Lambda_3 (G_1+G_2) = \Big ((\sigma(A_{G_1})\backslash \{r\})\cup (\sigma(A_{G_2})\backslash \{\tilde{r}\})\Big )\backslash \{r-m,r-2m\},\]
and \[
	\Lambda_4 (G_1+G_2) =\begin{cases}
		\{\frac{m\tilde{r}+nr-2mn}{m+n}\}, & \mbox{if }r-m\neq \tilde{r}-n, \\
		\emptyset, & \mbox{otherwise }.
	\end{cases}
\]
If $\lambda _{\mbox{min}}(G_1)$ and $\lambda _{\mbox{min}}(G_2)$ are the minimum eigenvalue of $G_1$ and $G_2$, respectively, by Theorem \ref{theorem1},
\begin{equation*}
    \qec(G_1+G_2)=-2-\max \{\lambda _{\mbox{min}}(G_1),\ \lambda _{\mbox{min}}(G_1),\ \frac{m\tilde{r}+nr-2mn}{m+n}\}.
\end{equation*}
\end{ex}
The formula for $\qec (K_2+P_3)$ is obtained in \cite[Theorem 3.9]{zakiyyah}. We now derive this from Theorem \ref{theorem1}.
\begin{ex}\label{example1}
Let $G_1=K_2$ and $G_2=P_3$. 
Here, $\sigma (A_{K_2})=\{1,-1\},\ \sigma (A_{P_3}-J)=\{0,-1,-2\},\ \sigma (A_{P_3})=\{0,\sqrt{2},-\sqrt{2}\}, \sigma _0 (A_{K_2})=\{-1\}$ and $\sigma _0 (A_{P_3})=\{0\}$. Since the matrix $A_{P_3}$ is
\begin{center}
$\begin{pmatrix}
	0 & 1 & 0\\
	1 & 0 & 1 \\
    0 & 1 & 0
\end{pmatrix}$,\ 
    $(A_{P_3}-\lambda I)^{-1}=\begin{pmatrix}
	\frac{-\lambda ^2+1}{\lambda ^3-2\lambda} & \frac{-1}{\lambda ^2-2} & \frac{-1}{\lambda ^3-2\lambda }\\
	\frac{-1}{\lambda ^2-2} & \frac{-\lambda}{\lambda ^2-2} & \frac{-1}{\lambda ^2-2} \\
    \frac{-1}{\lambda ^3-2\lambda } & \frac{-1}{\lambda ^2-2} & \frac{-\lambda ^2+1}{\lambda ^3-2\lambda}
\end{pmatrix}$
\end{center}
and $\langle e,(A_{P_3}-\lambda I)^{-1}e \rangle =-\frac{3\lambda +4}{\lambda ^2-2}$. So, the equation $2=(\lambda +3)\langle e,(A_{P_3}-\lambda I)^{-1}e \rangle$ has two solutions $-1$ and $-\frac{8}{5}$. This gives us
\begin{equation*}
    \Lambda_1 (K_2+P_3)=\{-1\},\ \Lambda_2 (K_2+P_3)=\emptyset ,\ \Lambda_3 (K_2+P_3)= \{0\}, \ \mbox{and}\ \Lambda_4 (K_2+P_3)=\{-\frac{8}{5}\}.
\end{equation*}
By Theorem \ref{theorem1},  $\qec (K_2+P_3)=-2-\min \{-\frac{8}{5}\}=-\frac{2}{5}$.
\end{ex}
We next derive QEC for several new classes of graphs as a consequence of our first main result.
\begin{ex}\label{example2}
Let $G_1=K_2$ and $G_2=P_4$. Now, $\sigma(A_{P_4})=\{\frac{\pm \sqrt{5}-1}{2}, \frac{\pm \sqrt{5}+1}{2}\},\ \sigma(A_{P_4}-J)=\{\frac{\pm \sqrt{5}-3}{2}, \frac{\pm \sqrt{5}-1}{2}\}$ and $\sigma _0(A_{P_4})=\{\frac{\pm \sqrt{5}-1}{2}\}$. Since $\langle e,(A_{P_4}-\lambda I)^{-1}e \rangle=-\frac{4\lambda +2}{\lambda ^2-\lambda -1}$, the equation $2=(\lambda +3)\langle e,(A_{P_4}-\lambda I)^{-1}e \rangle$ reduces to $3\lambda ^2 +6\lambda +2=0$ with two solutions $\frac{-3\pm \sqrt{3}}{3}$. Thus
\[
    \Lambda_1 (K_2+P_4)=\{-1\},\ \Lambda_2 (K_2+P_4)=\emptyset ,\ \Lambda_3 (K_2+P_4)= \left\{\frac{\pm \sqrt{5}-1}{2}\right\}, \ \mbox{and}\ \Lambda_4 (K_2+P_4)=\left\{\frac{-3\pm \sqrt{3}}{3}\right\}.
\]
By Theorem \ref{theorem1}, $\qec (K_2+P_4)=-2-\min \{\frac{-1- \sqrt{5}}{2},\frac{-3- \sqrt{3}}{3}\}=-2+\frac{1+\sqrt{5}}{2}=\frac{-3+\sqrt{5}}{2}$.
\end{ex}
We now extend the previous example to the case $G_1=K_m$. The graph $K_m+P_4$ plays an important role in reciprocal eigenvalue theory, having been studied in \cite[Definition 4]{barik} for constructing nonbipartite graphs with a strong reciprocal eigenvalue property. 
\begin{ex}\label{example3}
Let $G_1=K_m,m\geq 2$ and $G_2=P_4$. Then, $r=m-1, ~r-2m<-2, ~\sigma (A_{K_m})=\{m-1,-1\}$ and $\sigma _0 (A_{K_m})=\{-1\}$.  Also, $\alpha >-2$ for all $\alpha \in \sigma (A_{P_4})$. By a straightforward calculation one can verify that the roots of the equation $m=(\lambda +m+1)\langle e,(A_{P_4}-\lambda I)^{-1}e \rangle$ are $\frac{-(3m+6)\pm \sqrt{5m^2+12m+4}}{2(m+4)}$ and
\begin{equation}\label{inequality}
\frac{-(3m+6)- \sqrt{5m^2+12m+4}}{2(m+4)}\begin{cases}
> -\frac{1+\sqrt{5}}{2}, & \mbox{if }m=2, \\
<-\frac{1+\sqrt{5}}{2}\  \&\ >-2, & \mbox{if }m=3,4,5,\\
=-2, & \mbox{if }m=6,\\
<-2, & \mbox{if }m\geq 7.
\end{cases}
\end{equation}
Thus  $\Lambda_1 (K_m+P_4)=\{-1\}, ~\Lambda_2 (K_m+P_4)  =\emptyset,  ~\Lambda_3 (K_m+P_4) = \left\{\frac{\pm \sqrt{5}-1}{2}\right\}$, and $ \Lambda_4 (K_m+P_4) = \left\{\frac{-(3m+6)\pm \sqrt{5m^2+12m+4}}{2(m+4)}\right\}.$
Using Theorem \ref{theorem1}, we have
\[
\qec(K_m+P_4)=\begin{cases}
\frac{-3+\sqrt{5}}{2}, & \mbox{if }m=2, \\
\frac{-(m+10)+\sqrt{5m^2+12m+4}}{2(m+4)}, & \mbox{if }m\geq 3.
\end{cases}
\]
\end{ex}
\begin{rem}Note that $K_m+P_4$ is of QE-class if and only if $m\leq 6$. 
\end{rem}
In the next example, we compute $\qec(C_m + P_4)$. For $m = 3$, since $C_3 = K_3$, the value $\qec(K_3 + P_4)$ is given in Example~\ref{example3}. Thus, we consider $m \geq 4$.
\begin{ex}\label{example10}
Let $G_1=C_m$ with $m\geq 4$ and $G_2=P_4$. Then, $\sigma (A_{C_m})=\{2\cos \frac{2\pi k}{m}; k=1,2,\ldots ,m\}$ and $\sigma _0 (A_{C_m})=\sigma (A_{C_m})\backslash \{2\}$. Note that $r=2$,  $r-m\leq -2$ and $r-2m<-2$ for all $m\geq 4$. Thus
\begin{equation*}
 \Lambda _1 (C_m+P_4)=\begin{cases}
\{-2\}, & \mbox{if }m=4, \\
\emptyset , & \mbox{if }m\geq 5
\end{cases},
\quad \Lambda _2 (C_m+P_4)=\emptyset ,
\end{equation*}
\begin{equation*}
    \Lambda _3 (C_m+P_4)=\begin{cases}
\{0,\frac{\pm \sqrt{5}-1}{2}\}, & \mbox{if }m=4, \\
(\sigma (A_{C_m})\backslash \{2\})\cup \left\{\frac{\pm \sqrt{5}-1}{2}\right\} , & \mbox{if }m\geq 5.
\end{cases}
\end{equation*}
From \eqref{inequality} and using the fact that $-2\in \sigma _0(A_{C_6})$, we have
\begin{equation*}
   \Lambda _4 (C_m+P_4)=\begin{cases}
\{\frac{-(3m+6)\pm \sqrt{5m^2+12m+4}}{2(m+4)}\}, & \mbox{if }m\neq 6, \\
\{\frac{-(3m+6)+ \sqrt{5m^2+12m+4}}{2(m+4)} (>-2)\} , & \mbox{if }m= 6.
\end{cases} 
\end{equation*}
By Theorem \ref{theorem1},
\begin{center}
$\qec(C_m+P_4)=\begin{cases}
-2-(-2)=0, & \mbox{if }m=4,6 \\
-2-(\frac{-(3m+6)- \sqrt{5m^2+12m+4}}{2(m+4)})=\frac{-(m+10)+\sqrt{5m^2+12m+4}}{2(m+4)}, & \mbox{if }m=5,m\geq 7.
\end{cases}$  
\end{center}
\end{ex}
\begin{rem}
From the above formula, we can conclude that $C_m+P_4$ is of QE-class if and only if $m\leq 6$.
\end{rem}
We next turn our attention to the join of a complete graph $K_m$ and a complete bipartite graph $K_{p,q}$. If $p = q = 1$, then $K_m + K_{1,1}$ is a complete graph. Thus, without loss of generality, we assume that $p \geq 2$ and $q \geq 1$. For the complete bipartite graph $K_{p,q}$, $\sigma (A_{K_{p,q}})=\{0,\pm \sqrt{pq}\}$ with eigenvectors
$\{v^1,v^2,\ldots ,v^{p+q}\}$ defined as follows: 
for $1\leq j\leq p-1$,
\begin{equation*}
v^j_{i}=\begin{cases}
-1, & \mbox{if }i=1, \\
1, & \mbox{if }i=j+1,\\
0, & \mbox{Otherwise}.
\end{cases}
\end{equation*} For $p\leq j\leq p+q-2$ (with no contribution when $p=2$ and $q=1$),
\begin{equation*}
v^j_{i}=\begin{cases}
-1, & \mbox{if }i=p+1, \\
1, & \mbox{if }i=j+2,\\
0, & \mbox{Otherwise},
\end{cases}
\end{equation*}
\begin{center}
$v^{p+q-1}_i=\begin{cases}
\sqrt{\frac{q}{p}}, & \mbox{if }1\leq i\leq p, \\
1, & \mbox{Otherwise},
\end{cases}$ \ \ \ \ \ \ \ \ \ \ \ \ and  \ \ \ \ \ \ \ \ \ \ \ \ $v^{p+q}_i=\begin{cases}
-\sqrt{\frac{q}{p}}, & \mbox{if }1\leq i\leq p, \\
1, & \mbox{Otherwise}.
\end{cases}$
\end{center}

Define $P:=\left[v^1,v^2,\ldots,v^{p+q}\right]\in \mathbb{R}^{(p+q)\times (p+q)}$ and $D:=\diag\left(0,\ldots,0,\sqrt{pq},-\sqrt{pq}\right)\in \mathbb{R}^{(p+q)\times (p+q)}$.
Then $A_{K_{p,q}}=PDP^{-1}$ and 
$P^Te=(\langle v^1,e\rangle ,\langle v^2,e\rangle,\ldots ,\langle v^{p+q},e\rangle)^T=\Big (0,\ldots , 0,\sqrt{pq}+q,-\sqrt{pq}+q\Big )^T$. By a simple calculation, one can verify that
\begin{center}
$P^{-1}=\begin{pmatrix}
U & 0 \\
0 & X \\
Y & Z
\end{pmatrix},$
\end{center}
where $U=\begin{pmatrix}
	-\frac{1}{p}e & I-\frac{1}{p}J
\end{pmatrix}\in \mathbb{R}^{(p-1)\times p}$, $X=\begin{pmatrix}
	-\frac{1}{q}e & I-\frac{1}{q}J
\end{pmatrix}\in \mathbb{R}^{(q-1)\times q}$, $Y=\frac{1}{2\sqrt{pq}}\begin{pmatrix}
e^T \\
-e^T
\end{pmatrix}\in \mathbb{R}^{2\times p}$, and  $Z=\frac{1}{2q}\begin{pmatrix}
e^T\\
e^T
\end{pmatrix}\in \mathbb{R}^{2\times q}$.
Thus $P^{-1}e=\frac{1}{2}\Big (0,0,\ldots , 0,1+\sqrt{\frac{p}{q}},1-\sqrt{\frac{p}{q}}\Big )^T$ and
\begin{equation}\label{equation6}
   \langle e,(A_{K_{p,q}}-\lambda I)^{-1}e \rangle =\langle P^Te,(D-\lambda I)^{-1}P^{-1}e \rangle =\frac{1}{pq-\lambda ^2}(\lambda (p+q)+2pq).
\end{equation}

Using the above analysis, we now compute $\qec(K_m+K_{p,q})$.
\begin{ex}\label{exampletheorem}
Let $G_1=K_m,m\geq 1$ and $G_2=K_{p,q},\ p\geq 2,q\geq 1$. 
Then $\sigma (A_{p,q})=\{0,\pm \sqrt{pq}\}$,
\begin{equation*}
\sigma _0 (A_{p,q})=\begin{cases}
\{0,-p\}, & \mbox{if }p=q, \\
\{0\}, & \mbox{if }p\neq q,
\end{cases}
\end{equation*}
and $\sigma(A_{K_{p,q}}-J)=\{0,-p,-q\}$. Thus $\Lambda _1(K_m+K_{p,q})=\{-1\}$,
\begin{equation*}
    \Lambda _2 (K_m+K_{p,q})=\begin{cases}
\{-m-1\}, & \mbox{if }m=\sqrt{pq}-1, \\
\emptyset , & \mbox{Otherwise},
\end{cases}
\end{equation*}
and
\begin{equation*}
    \Lambda _3 (K_m+K_{p,q})=\begin{cases}
\{0\}, & \mbox{if }p\neq q, \\
\{0\}, & \mbox{if }p= q, m=p-1, \\
\{0,-p\} , & \mbox{Otherwise},
\end{cases}
\end{equation*}
Using \eqref{equation6}, the equation in $\Lambda _4(K_m+K_{p,q})$ reduces to
\begin{equation}\label{equation1}
    m = (\lambda +m+1)\frac{1}{pq-\lambda ^2}(\lambda (p+q)+2pq).
\end{equation}
For $p=q$, equation \eqref{equation1} has one root $-\frac{mp+2p}{m+2p}$. For $p\neq q$, the two roots of equation \eqref{equation1} are
\begin{equation}\label{equation2}
    \lambda _{\pm}=\frac{-(m+1)(p+q)-2pq\pm \sqrt{\{(m+1)(p+q)+2pq\}^2-4pq(m+p+q)(m+2)}}{2(m+p+q)}
\end{equation}
By \eqref{equation2}, it is easy to see that $\lambda _-<\lambda_+\leq -1$ and the equality holds if and only if either $p=1$, or $q=1$.
If $\lambda _-=-m-1$, then we have $(m+1)^2=pq$.
So, $\lambda _-=-m-1$, only when $m=\sqrt{pq}-1$. If $\lambda _-=-\sqrt{pq}$, then from \eqref{equation2}, we have $m=\sqrt{pq}-1$. Therefore,
\begin{equation*}
    \Lambda _4 (K_m+K_{p,q})=\begin{cases}
\{\lambda _{\pm}\}, & \mbox{if }p\neq q, q\neq 1, m\neq \sqrt{pq}-1 \\
\{\lambda _+\}, & \mbox{if }p\neq q, q\neq 1, m= \sqrt{pq}-1 \\
\{\lambda _-\}, & \mbox{if }p\neq q, q= 1, m\neq \sqrt{pq}-1 \\
\emptyset , & \mbox{if }p\neq q, q= 1, m= \sqrt{pq}-1 \\
\{-\frac{mp+2p}{m+2p}\}, & \mbox{if } p= q.
\end{cases}
\end{equation*}
By Theorem \ref{theorem1},
\begin{equation*}
\qec(K_m+K_{p,q})=\begin{cases}
-2-\min \{\lambda _\pm \}, & \mbox{if }p\neq q,q\neq 1, m\neq \sqrt{pq}-1, \\
-2-\min \{-m-1 (=\lambda _-),\lambda _+\}, & \mbox{if }p\neq q,q\neq 1, m= \sqrt{pq}-1, \\
-2-\min \{\lambda _-\}, & \mbox{if }p\neq q,q= 1, m\neq \sqrt{pq}-1, \\
-2-\min \{-m-1 (=\lambda _-)\}, & \mbox{if }p\neq q,q= 1, m= \sqrt{pq}-1, \\
-2-\min \{-p,-\frac{mp+2p}{m+2p}\}, & \mbox{if }p= q.
\end{cases}
\end{equation*}
Hence
\begin{equation*}
\qec(K_m+K_{p,q})=\begin{cases}
\frac{2(pq-2m)+(m-3)(p+q)+\sqrt{\{(m+1)(p+q)+2pq\}^2-4pq(m+p+q)(m+2)}}{2(m+p+q)}, & \mbox{if }p\neq q, \\
p-2, & \mbox{if }p= q.
\end{cases}
\end{equation*}
\end{ex}
\begin{rem}
Examples \ref{example1}--\ref{exampletheorem} demonstrate the necessity of computing all $\Lambda_i$ in the formula for the quadratic embedding constants in Theorem \ref{theorem1}. Note that
\[
K_m + K_{p,q} = K_{\underbrace{1,1,\ldots,1}_{m\ \text{times}}} + K_{p,q} = K_{p,q,\underbrace{1,1,\ldots,1}_{m\ \text{times}}}.
\]
Consequently, the formula in Example \ref{exampletheorem} follows from \cite[Theorem 1.1]{bata}.		
\end{rem}

\begin{ex}
Let $G_1=C_m,m\geq 5$ and $G_2=K_{p,q},\ p\geq 2,q\geq 1$. Without loss of generality assume that $p\geq q$. Recall that $\sigma (A_{C_m})=\{2\cos \frac{2\pi k}{m}; k=1,2,\ldots ,m\},\ \sigma _0 (A_{C_m})=\sigma (A_{C_m})\backslash \{2\},\ \sigma (A_{p,q})=\{0,\pm \sqrt{pq}\},\ \sigma(A_{K_{p,q}}-J)=\{0,-p,-q\}$ and
\begin{equation*}
\sigma _0 (A_{p,q})=\begin{cases}
\{0,-p\}, & \mbox{if }p=q, \\
\{0\}, & \mbox{if }p\neq q.
\end{cases}
\end{equation*}
 Also, note that $r-m=2-m< -2$ and $r-2m=2-2m<-2\ \forall \ m\geq 5$. So,
 \begin{equation*}
\Lambda _1(C_m+K_{p,q})=\begin{cases}
\{2-m\}, & \mbox{if }m=p+2 \mbox{ or, }q+2, \\
\emptyset , & \mbox{Otherwise},
\end{cases}
\end{equation*}
\begin{equation*}
\Lambda _2(C_m+K_{p,q})=\begin{cases}
\{2-2m\}, & \mbox{if }m=1+\frac{1}{2}\sqrt{pq}, \\
\emptyset , & \mbox{Otherwise},
\end{cases}
\end{equation*}
\begin{equation*}
\Lambda _3(C_m+K_{p,q})=\begin{cases}
(\sigma(A_{C_m})\backslash \{2\})\cup \{0\}, & \mbox{if }p\neq q, \\
(\sigma(A_{C_m})\backslash \{2\})\cup \{0\} , & \mbox{if }p=q\mbox{ and }m=p+2\mbox{ or }1+\frac{p}{2}, \\
(\sigma(A_{C_m})\backslash \{2\})\cup \{0,-p\} , & \mbox{Otherwise}.
\end{cases}
\end{equation*}
Using \eqref{equation6}, the equation in $\Lambda _4(C_m+K_{p,q})$ becomes
\begin{equation}\label{equation7}
    m = (\lambda +2m-2)\frac{1}{pq-\lambda ^2}(\lambda (p+q)+2pq).
\end{equation}
For $p=q$, the equation \eqref{equation7} has exactly one root $\frac{-3mp+4p}{m+2p}$. For $p\neq q$, the equation \eqref{equation7} reduces to
\begin{equation*}
    (m+p+q)\lambda ^2+2\big ((m-1)(p+q)+pq\big )\lambda +(3m-4)pq=0
\end{equation*}
with the roots
\begin{equation}\label{equation8}
   \lambda _{\pm}= \frac{-(m-1)(p+q)-pq\pm \sqrt{\{(m-1)(p+q)+pq\}^2-pq(m+p+q)(3m-4)}}{(m+p+q)}.
\end{equation}
Now, equating $\lambda _-$ or $\lambda _+$ with $2-m$, we have a relation
\begin{equation*}
    m^2-(4+p+q)m+(p+2)(q+2)=0.
\end{equation*}
So,
\begin{equation*}
2-m=\begin{cases}
\lambda _-, & \mbox{only if }m=p+2, \\
\lambda _+, & \mbox{only if }m=q+2. \\
\end{cases}
\end{equation*}
On the other hand, from $\lambda _-=2-2m$ or $\lambda _+=2-2m$, we get
\begin{equation*}
    4m^2-8m+4-pq=0,
\end{equation*}
and then
\begin{equation*}
2-2m= \lambda _-, \mbox{only if }m=1+\frac{1}{2}\sqrt{pq}.
\end{equation*}
In addition, $\lambda _+=2-m$ when $m=q+2=1+\frac{1}{2}\sqrt{pq}$. Also, equating $\lambda _+$ or $\lambda _-$ with $-\sqrt{pq}$, we have $m=1+\frac{1}{2}\sqrt{pq}$. A straightforward calculation shows that $\frac{-3mp+4p}{m+2p},\lambda _-<-2\ \forall \ m\geq 5,p\geq 2,q\geq 1$. Note that for $p>q$, $p+2$ never be equal to $q+2$ or $1+\frac{1}{2}\sqrt{pq}$. Let $\tilde{\Lambda _4} (C_m+K_{p,q})$ is the set obtained from $\Lambda _4 (C_m+K_{p,q})$ by not removing the elements of $\sigma (A_{C_n})$. Then,
\begin{equation*}
    \tilde{\Lambda _4} (C_m+K_{p,q})=\begin{cases}
\{\lambda _{\pm}\}, & \mbox{if }p\neq q,m\neq p+2,q+2,1+\frac{1}{2}\sqrt{pq}, \\
\{\lambda _+\}, & \mbox{if }p\neq q,m= p+2, \\
\{\lambda _-\}, & \mbox{if }p\neq q,m= q+2,m\neq 1+\frac{1}{2}\sqrt{pq}, \\
\{\lambda _+\}, & \mbox{if }p\neq q,m\neq  q+2,m= 1+\frac{1}{2}\sqrt{pq}, \\
\{\frac{-3mp+4p}{m+2p}\}, & \mbox{if } p= q, m\neq p+2, \\
\emptyset ,& \mbox{Otherwise.}
\end{cases}
\end{equation*}
By Theorem \ref{theorem1},
\begin{equation*}
\qec(C_m+K_{p,q})=\begin{cases}
-2-\min \{\sigma(A_{C_m})\backslash \{2\},\lambda _\pm\}, & \mbox{if }p\neq q, m\neq p+2,q+2,1+\frac{1}{2}\sqrt{pq}, \\
-2-\min \{2-m (=\lambda _-),\sigma(A_{C_m})\backslash \{2\},\lambda _+\}, & \mbox{if }p\neq q, m=p+2, \\
-2-\min \{2-m (=\lambda _+),\sigma(A_{C_m})\backslash \{2\},\lambda _-\}, & \mbox{if }p\neq q, m=q+2,m\neq 1+\frac{1}{2}\sqrt{pq} \\
-2-\min \{2-2m (=\lambda _-),\sigma(A_{C_m})\backslash \{2\},\lambda _+\}, & \mbox{if }p\neq q, m\neq q+2,m= 1+\frac{1}{2}\sqrt{pq} \\
-2-\min \{2-m (=\lambda _+),2-2m (=\lambda _-),\sigma(A_{C_m})\backslash \{2\}\}, & \mbox{if }p\neq q,m=q+2,1+\frac{1}{2}\sqrt{pq}, \\
-2-\min \{2-m,\sigma(A_{C_m})\backslash \{2\}\}, & \mbox{if }p= q, m= p+2, \\
-2-\min \{2-2m,\sigma(A_{C_m})\backslash \{2\},\frac{-3mp+4p}{m+2p}\}, & \mbox{if }p= q, m= 1+\frac{p}{2}, \\
-2-\min \{\sigma(A_{C_m})\backslash \{2\},-p,\frac{-3mp+4p}{m+2p}\}, & \mbox{if }p= q, m\neq p+2, 1+\frac{p}{2},
\end{cases}
\end{equation*}
Hence
\begin{equation*}
\qec(C_m+K_{p,q})=\begin{cases}
\frac{pq-2m+(m-3)(p+q)+\sqrt{\{(m-1)(p+q)+pq\}^2-pq(m+p+q)(3m-4)}}{(m+p+q)}, & \mbox{if }p\neq q, \\
p-2, & \mbox{if }p= q, m\leq p+2, \\
\frac{3mp-2m-8p}{m+2p}, & \mbox{if }p= q, m> p+2.
\end{cases}
\end{equation*}
\end{ex}
We next prove Theorem \ref{theorem2} -- we derive quadratic embedding constant of $K_{m_1,m_2,\ldots ,m_k}+G$, where $G$ is a graph of order $n$.  
The adjacency matrix of $K_{m_1,m_2,\ldots ,m_k}$  is a matrix of order $m_1+m_2+\cdots +m_k$ of the form
\[A_{K_{m_1,m_2,\ldots ,m_k}}=\begin{pmatrix}
0 & J & J & \hdots & J \\
J & 0 & J & \hdots & J \\
J & J & 0 & \hdots & J \\
\vdots & \vdots & \vdots & \ddots & \vdots \\
J & J & J & \hdots & 0
\end{pmatrix}.\]
For $A_{K_{m_1,m_2,\ldots ,m_k}+G}$, the  system of equations \eqref{firstequation}-\eqref{fourthequation} reduces to
\begin{eqnarray}
	(J+\lambda I)x^i&=&-\frac{\mu}{2}\  e \hbox{~~~for~}\ i\in[k] \label{cfirstequation} \\
	(A_G-J-\lambda I)y&=&\frac{\mu}{2}\  e, \label{csecondequation} \\
	\sum _{i=1}^k\langle e,x^i \rangle + \langle e,y \rangle& =&0, \label{cthirdequation} \\
    \sum _{i=1}^k\langle x^i,x^i \rangle + \langle y,y \rangle& =&1, \label{cfourthequation}
\end{eqnarray}
where $x=({x^1}^T,{x^2}^T,\ldots ,{x^k}^T)^T$ and $x^i\in \mathbb{R}^{m_i}$ for all $i\in[k]$ and $y \in \mathbb{R}^n$. 

\begin{proof}[Proof of Theorem \ref{theorem2}]
Since $K_{m_1,m_2,\ldots ,m_k}+G$ is not a complete graph, by equation \eqref{newgraphjoinqec} and Proposition \ref{qeccomplete}, $\qec (K_{m_1,m_2,\ldots ,m_k}+G)=-2-\min \left(\Lambda (K_{m_1,m_2,\ldots ,m_k}+G)\cap(-\infty,-1)\right)$. To complete the proof, it is sufficient to show that 
\[\Lambda (K_{m_1,m_2,\ldots ,m_k}+G)\cap(-\infty,-1)=\bigcup\limits_{i=1}^4(\Lambda_i (K_{m_1,m_2,\ldots ,m_k}+G)\cap (-\infty , -1))\cup \{-m_{i_p}: p\in[q], a_p\geq 2, m_{i_p}\neq 1\}.\]
Without loss of generality, assume that $m_1\geq m_2\geq \cdots \geq m_k$. We first claim that  if $a_p\geq 2$ then $-m_{i_p}\in \Lambda (K_{m_1,m_2,\ldots ,m_k}+G)$; in other words, $-m_{i_p}$ appears in  the solution of  \eqref{cfirstequation}-\eqref{cfourthequation}.  Let $a_p\geq 2$ for some $p\in[q]$ and consider two possible cases.

\noindent {\bf Case I.} $a_p$ is even:  Choose $\mu =0,~y=0$ and $x=({x^1}^T,{x^2}^T,\ldots ,{x^k}^T)^T\in \mathbb{R}^{m_1+m_2+\cdots m_k}$ such that
\begin{center}
$x^j=\begin{cases}
\frac{1}{\sqrt{a_pm_{i_p}}}e, & \mbox{if }\sum\limits_{g=1}^{p-1}a_g+1\leq j\leq \sum\limits _{g=1}^{p-1}a_g+\frac{a_p}{2}, \\
-\frac{1}{\sqrt{a_pm_{i_p}}}e, & \mbox{if }\sum\limits _{g=1}^{p-1}a_g+\frac{a_p}{2}+1\leq j\leq \sum \limits_{g=1}^{p-1}a_g+a_p,\\
0, & \mbox{otherwise.}
\end{cases}$
\end{center}
Then, $(x^1,\ldots ,x^k,y,-m_{i_p} ,\mu )$ is a solution of (\ref{cfirstequation})-(\ref{cfourthequation}).

\noindent {\bf Case II.} $a_p$ is odd: Take $\mu =0$, $y=0$ and
\begin{center}
$x^j=\begin{cases}
\frac{1}{\sqrt{(a_p-1)m_{i_p}}}e, & \mbox{if }\sum\limits _{g=1}^{p-1}a_g+1\leq j\leq \sum\limits _{g=1}^{p-1}a_g+\frac{a_p-1}{2}, \\
-\frac{1}{\sqrt{(a_p-1)m_{i_p}}}e, & \mbox{if }\sum\limits _{g=1}^{p-1}a_g+\frac{a_p-1}{2}+1\leq j\leq \sum\limits _{g=1}^{p-1}a_g+(a_p-1),\\
0, & \mbox{otherwise.}
\end{cases}$
\end{center}
Then $(x^1,\ldots ,x^k,-m_{i_p},\mu )$ is a solution of (\ref{cfirstequation})-(\ref{cfourthequation}).

To classify all  other elements of $\Lambda (K_{m_1,m_2,\ldots ,m_k}+G)\cap(-\infty,-1)$, we now split the rest of the proof into several lemmas.

\begin{lem}\label{firstlemma}
Let $\lambda = -m_{i_{p}}$ for some $p\in[q]$ such that $a_p=1$. Then $\lambda$ appears in the solutions of \eqref{cfirstequation}-\eqref{cfourthequation} if and only if $\lambda \in \sigma(A_G-J)$.
\end{lem}
\begin{proof}
Suppose $(x,y,-m_{i_{p}},\mu)$ is a solution of \eqref{cfirstequation}-\eqref{cfourthequation}, where $x=({x^1}^T,{x^2}^T,\ldots ,{x^k}^T)^T\in \mathbb{R}^{m_1+\cdots +m_k},y\in \mathbb{R}^n,\mu \in \mathbb{R}$ . From \eqref{cfirstequation}, one can derive
\begin{equation}\label{xi}
\mu =0,\ x^{i_{p}}=ce \mbox{ for arbitrary }c\in \mathbb{R},\  x^j=0 \ \mbox{ for all } \ j\in[k] \mbox{ and }j\neq i_{p}.
\end{equation}
Thus the equations \eqref{csecondequation}-\eqref{cfourthequation} reduces to the following
\begin{equation}\label{threeequation}
(A_G-J+m_{i_{p}} I)y=0,\ cm_{i_{p}}+\langle e,y \rangle =0,\ c^2m_{i_{p}}+\langle y,y \rangle =1.
\end{equation}
If $y=0$, then $c=0$ from the second identity of \eqref{threeequation} and hence the third identity does not hold. Thus $y\neq 0$ and $-m_{i_{p}}\in \sigma(A_G-J)$.

Conversely, suppose $-m_{i_{p}} \in \sigma(A_G-J)$ for some $p\in[q]$ with $a_p=1$. Then, there exists an $0\neq y_0\in \mathbb{R}^n$ such that $(A_G-J+m_{i_{p}} I)y_0=0$. Define $x:=({x^1}^T,{x^2}^T,\ldots ,{x^k}^T)^T$ such that $x^j\in \mathbb{R}^{m_j}$ and
\begin{center}
$x^j:=\begin{cases}
ce, & \mbox{if }j=i_{p}, \\
0, & \mbox{Otherwise.}
\end{cases}$
\end{center}
By choosing appropriate values of $\gamma $ and $c$, one can show that $(x,y=\gamma y_0,-m_{i_{p}}, 0)$ is a solution of \eqref{cfirstequation}-\eqref{cfourthequation}.
\end{proof}
Since $\lambda<-1$, before we proceed to the next lemma, note that for $\lambda \neq -m_1,-m_2,\ldots , -m_k$, equation \eqref{cfirstequation} implies that
\begin{equation}\label{cnewfirstequation}
x^i=-\frac{1}{\lambda +m_i}\frac{\mu}{2}e,\ \  i\in [k]
\end{equation}
and equations \eqref{csecondequation}, \eqref{cthirdequation} implies that
\begin{equation}\label{cnewsecondequation}
(A_G-\lambda I)y=\frac{\mu}{2}e+Jy=\frac{\mu}{2}e+(-\sum _{i=1}^k\langle e,x^i \rangle)e=\Big (1+\sum _{p=1}^q\frac{a_pm_{i_p}}{\lambda +m_{i_p}} \Big )\frac{\mu}{2}e=P(\lambda)\frac{\mu}{2}e,
\end{equation}
where $P$ is defined in \eqref{eqplambda}.
\begin{lem}\label{secondlemma}
Let $\lambda = \lambda _i$ for some $p\in[q]$. Then $\lambda$ appears in the solutions of \eqref{cfirstequation}-\eqref{cfourthequation} if and only if $\lambda  \in \sigma(A_G)$.
\end{lem}
\begin{proof}
Let $\lambda = \lambda _i$ for some $p\in[q]$. Since $\lambda_i$ is zero of $P$,
\[P(\lambda)=1+\sum _{p=1}^q\frac{a_pm_{i_p}}{\lambda  +m_{i_p}}=0 .\]
 Suppose $\lambda$ appears in the solutions of \eqref{cfirstequation}-\eqref{cfourthequation}. Then there exist $x=({x^1}^T,{x^2}^T,\ldots ,{x^k}^T)^T\in \mathbb{R}^{m_1+\cdots +m_k},y\in \mathbb{R}^n,\mu \in \mathbb{R}$ such that $(x,y,\lambda,\mu)$ satisfies \eqref{cnewfirstequation}, \eqref{cnewsecondequation}, \eqref{cthirdequation} and \eqref{cfourthequation}. By \eqref{cnewsecondequation}, $(A_G-\lambda I)y=0$. We next show that $y\neq 0$.
If not, then from \eqref{cnewfirstequation} and \eqref{cthirdequation}, we have
\begin{equation*}
-\frac{\mu}{2}\sum _{p=1}^q\frac{a_pm_{i_p}}{\lambda +m_{i_p}}=0.
\end{equation*}
This implies $\mu=0$ and so $x^j=0$ for all $j\in[k]$, which contradicts \eqref{cfourthequation}. Thus $y\neq 0$ and from \eqref{cnewsecondequation}, $\lambda \in \sigma(A_G)$.

Conversely, suppose $\lambda \in \sigma(A_G)$. Then there exists an $0\neq y_0\in \mathbb{R}^n$ such that
\begin{equation*}
(A_G-\lambda I)y_0=0.
\end{equation*}
Choose $x^j=-\frac{1}{\lambda +m_j}\frac{\mu}{2}e$ for $j\in [k]$ and $c,\mu$ such that
\begin{eqnarray*}
	-\frac{\mu}{2}\Big (\sum_{p=1}^q\frac{a_pm_{i_p}}{\lambda+m_{i_p}}\Big )+c\langle e,y_0 \rangle&=&0, \\
    \Big (\frac{\mu}{2} \Big )^2\Big (\sum_{p=1}^q\frac{a_pm_{i_p}}{(\lambda +m_{i_p})^2} \Big ) +c^2\langle y_0,y_0 \rangle & =&1.
\end{eqnarray*}
Then $(x,y,\lambda,\mu)$ is a solution of \eqref{cfirstequation}-\eqref{cfourthequation}, where $x=({x^1}^T,{x^2}^T,\ldots ,{x^k}^T)^T\in \mathbb{R}^{m_1+\cdots +m_k},y=cy_0\in \mathbb{R}^n$.
\end{proof}
\begin{lem}\label{fourthlemma}
Let $\lambda \in \sigma(A_G)\backslash \{-m_{i_p},\lambda _p: p\in[q]\}$. Then $\lambda$ appears in the solutions of \eqref{cfirstequation}-\eqref{cfourthequation} if and only if $\lambda \in \sigma_0(A_G)$.
\end{lem}
\begin{proof}
Suppose $\lambda \in \sigma_0(A_G)$. Then there exists $0\neq y\in \mathbb{R}^n$ such that
\begin{equation*}
A_Gy=\lambda y \ \ \mbox{and}\ \  \langle e,y \rangle =0.
\end{equation*}
Then $(x^1=0,x^2=0,\ldots ,x^k=0,\frac{y}{\parallel y\parallel _2},\lambda ,\mu =0)$ is a solution of \eqref{cfirstequation}-\eqref{cfourthequation}.

To prove the converse, let $(x_1,\ldots,x_k,y,\lambda, \mu)\in \mathbb{R}^{m_1}\times \cdots\times \mathbb{R}^{m_k} \times \mathbb{R}^n \times \mathbb{R} \times \mathbb{R}$ be a solution of the system of equations \eqref{cfirstequation}-\eqref{cfourthequation}. To show $\lambda \in \sigma_0(A_G)$,  we now consider two cases.

\noindent {\bf Case I.} $\mu =0$: By equations \eqref{cnewfirstequation}, \eqref{cnewsecondequation}, we have 
\begin{eqnarray*}
	x^i&=&0 \hbox{~~for~~}  i\in[k] \\
	A_Gy&=&\lambda y.
\end{eqnarray*}
From \eqref{cthirdequation} and \eqref{cfourthequation}, we have $\langle e,y \rangle =0$ and $y\neq 0$. Thus $\lambda \in \sigma_0(A_G)$.

\noindent {\bf Case II.} $\mu \neq 0$: Since $\lambda$ is not a zero of $P$ and $y$ is a solution of the equation \eqref{cnewsecondequation}, $\rk(A_{G}-\lambda I,~e)= \rk(A_{G}-\lambda I)$. Thus $\rk(A_{G}-\lambda I,~e)<n$,  since $A_{G} \in \mathbb{R}^{n \times n}$ and $\lambda \in \sigma (A_{G})$. Hence there exists a $z \in \mathbb{R}^n$ with $\langle z, z\rangle=1$ such that $z^T(A_{G}-\lambda I,~e)=0$.  Since $(A_{G}-\lambda I)$ is a symmetric matrix,
\begin{equation*}
	(A_{G}-\lambda I)z=0\ \ \  \mbox{and}\ \ \  \langle e,z\rangle =0.
\end{equation*}
Thus $\lambda \in \sigma_0(A_G)$ and concludes the proof.
\end{proof}
\begin{lem}\label{thirdlemma}
	Let $\lambda \in \mathbb{R}\backslash ( \{-m_{i_p},\lambda _p: p\in[q]\}\cup \sigma(A_G) )$. Then $\lambda$ appears in the solutions of \eqref{cfirstequation}-\eqref{cfourthequation} if and only if $\lambda$ is the solutions of
	\begin{equation}\label{rnew3equation}
		P(\lambda)\langle e,(A_G-\lambda I)^{-1}e \rangle-(P(\lambda)-1)=0	
	\end{equation}
\end{lem}
\begin{proof}
	Suppose $\lambda$ is in the solutions of \eqref{cfirstequation}-\eqref{cfourthequation}. Then, there exist $(x^1,\ldots ,x^k,y,\lambda ,\mu)$ that satisfy \eqref{cnewfirstequation}, \eqref{cnewsecondequation}, \eqref{cthirdequation} and \eqref{cfourthequation}. From \eqref{cnewsecondequation}, we obtain
	\begin{equation}\label{inversesecondequation}
		y=P(\lambda)\frac{\mu}{2}(A_G-\lambda I)^{-1}e.
	\end{equation}
	Applying \eqref{cnewfirstequation}, \eqref{inversesecondequation} in \eqref{cthirdequation}, we have
	\[ \frac{\mu}{2}\left(P(\lambda) \langle e,(A_G-\lambda I)^{-1}e \rangle-(P(\lambda)-1) \right)=0.\]
	If $\mu =0$, then from \eqref{cnewfirstequation}, $x^i=0$ for all $i\in [k]$ and from \eqref{inversesecondequation}, $y=0$ which contradicts \eqref{cfourthequation}. Thus $\mu \neq 0$ and $\lambda $ satisfies \eqref{rnew3equation}.
	
	To prove the converse, assume that $\lambda$ satisfies \eqref{rnew3equation}. Take $x^i$ for $i\in[k]$ and $y$ as in \eqref{cnewfirstequation} and \eqref{inversesecondequation} respectively, and $\mu$ so that equation \eqref{cfourthequation} holds. Then $(x^1,\ldots ,x^k,y,\lambda ,\mu)$ is a solution of \eqref{cfirstequation}-\eqref{cfourthequation}.
\end{proof}

Using the above four lemmas, one can conclude that \[\Lambda (K_{m_1,m_2,\ldots ,m_k}+G)\cap(-\infty,-1)=\bigcup\limits_{i=1}^4(\Lambda_i (K_{m_1,m_2,\ldots ,m_k}+G)\cap (-\infty , -1))\cup \{-m_{i_p}: p\in[q], a_p\geq 2, m_{i_p}\neq 1\}.\]
This completes the proof.
\end{proof}
As an immediate consequence of Theorem \ref{theorem2}, we now derive QEC for several classes of graphs.
\begin{ex}\label{example5}
Let $m_1=2,m_2=1$ and $G=P_3$. Then $\sigma (A_{P_3})=\{0,\pm \sqrt{2}\},\ \sigma (A_{P_3}-J)=\{0,-1,-2\}$ and $\sigma_0 (A_{P_3})=\{0\}$. The function $P(\lambda)=1+\frac{1}{\lambda +1}+\frac{2}{\lambda +2}=\frac{\lambda ^2+6\lambda +6}{(\lambda +1)(\lambda +2)}$ has two zeros $\lambda =-3\pm \sqrt{3}$. Recall that $\langle e,(A_{P_3}-\lambda I)^{-1}e \rangle =-\frac{3\lambda +4}{\lambda ^2-2}$. So, the equation $P(\lambda)\langle e,(A_{P_3}-\lambda I)^{-1}e \rangle=P(\lambda)-1$ equivalent to the equation $3\lambda ^3+13\lambda ^2+18\lambda +8=0$ with three roots $-1,-\frac{4}{3},-2$. Thus
\begin{equation*}
    \Lambda _1 (K_{1,2}+P_3)=\{-1,-2\},\ \Lambda _2 (K_{1,2}+P_3)=\emptyset,\ \Lambda _3 (K_{1,2}+P_3)=\{0\},\ \Lambda _4 (K_{1,3}+P_3)=\{-\frac{4}{3}\}. 
\end{equation*}
By Theorem \ref{theorem2},
\begin{equation*}
\qec(K_{1,2}+P_3)=-2-\min \{-\frac{4}{3},-2\}=0.
\end{equation*}
\end{ex}
\begin{ex}\label{example9}
	Let $m_1=1,m_2=1$ and $G=K_{1,9}$. Then $\sigma(A_{K_{1,9}})=\{0,3,-3\},\ \sigma(A_{K_{1,9}}-J)=\{0,-1,-9\}$, $\sigma _0(A_{K_{1,9}})=\{0,-3\}$, and $P(\lambda)=\frac{\lambda +3}{\lambda +1}$. From equation \eqref{equation6}, we know that $\langle e,(A_{K_{1,9}}-\lambda I)^{-1}e \rangle =\frac{10\lambda +18}{9-\lambda ^2}$. So, the equation in $\Lambda_4 (K_{1,1}+K_{1,9})$ gives us \[\Big (\frac{\lambda +3}{\lambda +1}\Big )\Big (\frac{10\lambda +18}{9-\lambda ^2} \Big )=\frac{2}{\lambda +1} \quad \Rightarrow \lambda =-1.\]
	This gives us
	\begin{equation*}
		\Lambda_1 (K_{1,1}+K_{1,9})= \{-1\} ,\ \Lambda_2 (K_{1,1}+K_{1,9})= \{-3\} ,\ \Lambda_3 (K_{1,1}+K_{1,9})=\{0\},\ \Lambda _4(K_{1,1}+K_{1,9})=\emptyset.
	\end{equation*}
	By Theorem \ref{theorem2},
	\begin{equation*}
		\qec(K_{1,1}+K_{1,9})=-2-\min \{-3\}= 1.
	\end{equation*}
\end{ex}
\begin{ex}\label{example7}
	Let $m_1=3,m_2=3,m_3=1$ and $G=P_3$. Then $\sigma (A_{P_3})=\{0,\pm \sqrt{2}\},\ \sigma (A_{P_3}-J)=\{0,-1,-2\}$, $\sigma_0 (A_{P_3})=\{0\}$ and $P(\lambda)=\frac{\lambda ^2+11\lambda +12}{(\lambda +1)(\lambda +3)}$ has zeross $\frac{-11\pm \sqrt{73}}{2}$. The equation $P(\lambda)\langle e,(A_{P_3}-\lambda I)^{-1}e \rangle=P(\lambda)-1$ reduces to the equation $5\lambda ^3+23\lambda ^2+33\lambda +15=0$ with the roots $-1,\frac{-9\pm \sqrt{6}}{5}$. So,
	\begin{equation*}
		\Lambda _1 (K_{1,3,3}+P_3)=\{-1\} ,\ \Lambda _2 (K_{1,3,3}+P_3)=\emptyset ,\ \Lambda_3 (K_{1,3,3}+P_3)= \{0\},\ \Lambda_4 (K_{1,3,3}+P_3)=\{\frac{-9\pm \sqrt{6}}{5}\}.
	\end{equation*}
	By Theorem \ref{theorem2},
	\begin{equation*}
		\qec(K_{1,3,3}+P_3)=-2-\min \{\frac{-9\pm \sqrt{6}}{5},-3\}=1.
	\end{equation*}
\end{ex}
\begin{rem}
	 Since $K_{m_1,m_2,\ldots,m_k} + K_{n_1,n_2,\ldots,n_l} = K_{m_1,m_2,\ldots,m_k,n_1,n_2,\ldots,n_l}$, the formulas in Examples \ref{example5}, \ref{example9}, and \ref{example7} can also be verified using \cite[Theorem 1.1]{bata}.	
\end{rem}
\begin{ex}
Let $m_1=4,m_2=1$ and $G=C_5$. Then $\sigma (A_{C_5})=\{2,\frac{\pm \sqrt{5}-1}{2}\},\ \sigma (A_{C_5}-J)=\{-3,\frac{\pm \sqrt{5}-1}{2}\}$ and $\sigma_0 (A_{C_5})=\{\frac{\pm \sqrt{5}-1}{2}\}$. Now, $P(\lambda)=1+\frac{1}{\lambda +1}+\frac{4}{\lambda +4}=\frac{\lambda ^2+10\lambda +12}{(\lambda +1)(\lambda +4)}$ has zeros $\lambda =-5\pm \sqrt{13}$. We have $\langle e,(A_{C_5}-\lambda I)^{-1}e \rangle =-\frac{5}{\lambda -2}$ and the equation $P(\lambda)\langle e,(A_{C_5}-\lambda I)^{-1}e \rangle=P(\lambda)-1$ has solutions $\frac{-12\pm \sqrt{34}}{5}$. So,
\begin{equation*}
   \Lambda _1 (K_{1,4}+C_5)=\emptyset ,\ \Lambda _2 (K_{1,4}+C_5)=\emptyset,\ \Lambda _3 (K_{1,4}+C_5)=\{\frac{\pm \sqrt{5}-1}{2}\},\ \Lambda _4 (K_{1,4}+C_5)=\{\frac{-12\pm \sqrt{34}}{5}\}. 
\end{equation*}
By Theorem \ref{theorem2},
\begin{equation*}
\qec(K_{1,4}+C_5)=-2-\min \{\frac{- \sqrt{5}-1}{2},\frac{-12\pm \sqrt{34}}{5}\}=\frac{2+ \sqrt{34}}{5}.
\end{equation*}
\end{ex}
\begin{ex}\label{example8}
Let $m_1=1,m_2=1$ and $G=P_4$. Recall that $\sigma(A_{P_4})=\{\frac{\pm \sqrt{5}-1}{2}, \frac{\pm \sqrt{5}+1}{2}\},\ \sigma(A_{P_4}-J)=\{\frac{\pm \sqrt{5}-3}{2}, \frac{\pm \sqrt{5}-1}{2}\}$ and $\sigma _0(A_{P_4})=\{\frac{\pm \sqrt{5}-1}{2}\}$. The function $P(\lambda)=\frac{\lambda +3}{\lambda +1}$ has  a zero $\lambda =-3$. Now, recall that $\langle e,(A_{P_4}-\lambda I)^{-1}e \rangle =-\frac{4\lambda +2}{\lambda ^2-\lambda -1}$. So, the equation in $\Lambda_4 (K_{1,1}+P_4)$, that is, $3\lambda ^2+6\lambda +2=0$ has roots $\frac{-3\pm \sqrt{3}}{3}$. Thus
\begin{equation*}
    \Lambda_1 (K_{1,1}+P_4)= \emptyset ,\ \Lambda_2 (K_{1,1}+P_4)= \emptyset, \ \Lambda_3 (K_{1,1}+P_4)=\{\frac{\pm \sqrt{5}-1}{2}\}, \Lambda_4 (K_{1,1}+P_4)= \{\frac{-3\pm \sqrt{3}}{3}\}.
\end{equation*}
By Theorem \ref{theorem2},
\begin{equation*}
\qec(K_{1,1}+P_4)=-2-\min \{\frac{-1- \sqrt{5}}{2},\frac{-3\pm \sqrt{3}}{3}\}= \frac{-3+\sqrt{5}}{2}
\end{equation*}
and  $K_{1,1}+P_4$ is of QE class.
\end{ex}
The above examples demonstrate that each $\Lambda_i$ plays a crucial role in computing the quadratic embedding constant in Theorem \ref{theorem2}. On the other hand, Example \ref{example7} highlights the necessity of considering the set $\{-m_{i_p} : p \in [q],\ a_p \geq 2,\ m_{i_p} \neq 1\}$. Now, we study the quadratic embedding constants of join graphs $K_{m,1}+\tilde{G}$, where $\tilde{G}$ is either a friendship graph $F_n=K_1+nK_2$ or a wheel graph $W_n=K_1+C_n$.
\begin{ex}
Let $m_1=m,m_2=1,m_3=1$ and $G=nK_2$. Then $K_{m,1,1}+nK_2=K_{m,1}+F_n$. Now, $\sigma (A_{nK_2})=\{-1,1\},\ \sigma (A_{nK_2}-J)=\{-1,1-2n,1\}$ and $\sigma_0 (A_{nK_2})=\{-1\}$. So,
\begin{eqnarray*}
	\Lambda _1 (K_{m,1,1}+nK_2)\cap (-\infty ,-1)&=& \begin{cases}
\{1-2n\} , & \mbox{if }m\neq 1, m=2n-1 \\
\emptyset , & \mbox{Otherwise.}
\end{cases} \\
	\Lambda _2 (K_{m,1,1}+nK_2)\cap (-\infty ,-1)&=&\emptyset ,\\
    \Lambda _3 (K_{m,1,1}+nK_2)\cap (-\infty ,-1)&=&\emptyset .
\end{eqnarray*}
Now, \[P(\lambda)= \frac{\lambda ^2+(2m+3)\lambda +4m}{(\lambda +m)(\lambda +1)},\ P(\lambda)-1=\frac{(m+2)\lambda +3m}{(\lambda +m)(\lambda +1)}\]
and $\langle e,(A_{nK_2}-\lambda I)^{-1}e \rangle =-\frac{2n}{\lambda -1}$. Then $P(\lambda)\langle e,(A_{nK_2}-\lambda I)^{-1}e \rangle =P(\lambda )-1$ gives us
\begin{equation}\label{equation9}
    (m+2+2n)\lambda ^2+2\lambda (m-1+2mn+3n)+m(8n-3).
\end{equation}
The roots of the equation \eqref{equation9} are
\begin{equation*}
     \lambda _{\pm}=\frac{-(m-1+2mn+3n)\pm \sqrt{(m-1+2mn+3n)^2-m(8n-3)(m+2+2n)}}{(m+2+2n)}.
\end{equation*}
It is clear that the roots of $P(\lambda)=0$ will never be the roots of \eqref{equation9}. And $-m$ will be equal to $\lambda _-$ or $\lambda _+$ if and only if
\begin{equation*}
m^2-2nm+(2n-1)=0.
\end{equation*}
So, 
\begin{equation*}
-m=\begin{cases}
\lambda _-, & \mbox{only if }m=2n-1, m\neq 1 \\
\lambda _+, & \mbox{only if }m=1. \\
\end{cases}
\end{equation*}
By a simple calculation, one can verify that $\lambda _+\leq -1$, and equality holds if and only if either $m=1$ or $n=1$. Therefore,
\begin{equation*}
    \Lambda _4 (K_{m,1,1}+nK_2)\cap (-\infty ,-1)=\begin{cases}
\{\lambda _-\}, & \mbox{if }m=1, \\
\{\lambda _+\}, & \mbox{if }m\neq 1,m= 2n-1, \\
\{\lambda _\pm \} ,& \mbox{Otherwise.}
\end{cases}
\end{equation*}
By Theorem \ref{theorem2},
\begin{equation*}
\qec(K_{m,1,1}+nK_2)=\begin{cases}
-2-\min \{\lambda _-\}, & \mbox{if }m= 1, \\
-2-\min \{1-2n (\lambda _-),\lambda _+\}, & \mbox{if }m\neq 1, m= 2n-1, \\
-2-\min \{\lambda _{\pm }\}, & \mbox{Otherwise}.
\end{cases}
\end{equation*}
Thus,
\begin{equation*}
    \qec(K_{m,1,1}+nK_2)= \frac{2mn-n-m-5+ \sqrt{(m-1+2mn+3n)^2-m(8n-3)(m+2+2n)}}{(m+2+2n)}.
\end{equation*}
\end{ex}
\begin{rem}
It follows from the above formula that $K_{m,1} + F_n$ belongs to the QE-class if and only if $(m,n) \in \{(1,1), (2,1), (1,2)\}$.
\end{rem}
\begin{ex}\label{example11}
Let $m_1=m\geq 2,m_2=1,m_3=1$ and $G=C_n,n\geq 4$. Then $K_{m,1,1}+C_n=K_{m,1}+W_n$. Now, $\sigma (A_{C_m})=\{2\cos \frac{2\pi k}{m}; k=1,2,\ldots ,m\},\ \sigma _0 (A_{C_m})=\sigma (A_{C_m})\backslash \{2\}$. Also, if $\alpha \in \sigma (A_{C_n}-J)$, then either $\alpha =-(n-2)$ or $| \alpha |\leq 2$. In addition, $-2\in \sigma (A_{C_n}-J)$, only if $n$ is even. 
\begin{eqnarray*}
	\Lambda _1 (K_{m,1,1}+C_n)\cap (-\infty ,-1)&=& \begin{cases}
\{-2\} , & \mbox{if }m=2, n\mbox{ is even} \\
\{-m\}  , &\mbox{if }m\neq 2,n=m+2, \\
\emptyset  , &\mbox{Otherwise.}
\end{cases}
\end{eqnarray*}
It is clear that if $\lambda \in \Lambda _2 (K_{m,1,1}+C_n)\cap (-\infty ,-1)$, or $\lambda \in \Lambda _3 (K_{m,1,1}+C_n)\cap (-\infty ,-1)$, in both the case $|\lambda |\geq -2$. Now, $\langle e,(A_{nK_2}-\lambda I)^{-1}e \rangle =-\frac{n}{\lambda -2}$. Then $P(\lambda)\langle e,(A_{nK_2}-\lambda I)^{-1}e \rangle =P(\lambda )-1$ reduces to
\begin{equation}\label{equation10}
    (m+n+2)\lambda ^2+\lambda (m+3n+2mn-4)+2m(2n-3).
\end{equation}
The roots of the equation \eqref{equation9} are
\begin{equation*}
     \lambda _{\pm}=\frac{-(m+3n+2mn-4)\pm \sqrt{(m+3n+2mn-4)^2-8m(2n-3)(m+n+2)}}{2(m+n+2)}.
\end{equation*}
By a simple calculation, one can verify that $\lambda _-\leq -2\ \forall \ m\geq 2,n\geq 4$ and the equality holds if and only if $m=2,n=4$. And $\lambda _-=-m$ if and only if $n=m+2$. So, $\lambda _-\in \Lambda _4 (K_{m,1,1}+C_n)\cap (-\infty ,-1)$ only if $n\neq m+2$. Now,
\begin{equation*}
\min \Big (\bigcup\limits_{i=1}^4(\Lambda_i (K_{m,1,1}+C_n)\cap (-\infty , -1)) \Big )=\begin{cases}
-2, & \mbox{if }m= 2,n=4, \\
-m, & \mbox{if }m\neq 2,n=m+2, \\
\lambda _-, & \mbox{Otherwise.}
\end{cases}
\end{equation*}
By Theorem \ref{theorem2},
\begin{equation*}
    \qec (K_{m,1,1}+C_n)=\begin{cases}
m-2, & \mbox{if }n=m+2, \\
\frac{(2mn-3m-n-12)+\sqrt{(m+3n+2nm-4)^2-8m(2n-3)(n+m+2)}}{2(m+n+2)}, & \mbox{Otherwise.}
\end{cases}
\end{equation*}
\end{ex}
\begin{rem}
From the above formula, it follows that $K_{m,1}+W_n$ is of QE-class if and only if $m=2,n=4$.
\end{rem}
\begin{rem}
In example \ref{example11}, when $m=1$, the formula for $\qec (K_{1,1,1}+C_n)=\qec (K_3+C_n)$ can be obtained from \cite[Example 3.6]{lou}. In addition, when $n=3$, the formula for $\qec (K_{m,1,1}+C_3)=\qec (K_{m,1,1,1,1,1})$ can be obtained from \cite[Theorem 1.1]{bata}.
\end{rem}
\section{Quadratic embedding constant and Cartesian Product of Graphs}\label{cartesianproductsection}
In this section, we first prove Proposition \ref{qecrelation}. We then prove Theorems \ref{theorem4}  and \ref{theorem5} -- for a arbitrary connected graph $G$, we derive a formula for the quadratic embedding constant of the Cartesian product of a complete graph $K_m$ and $G$  as well as QEC of the Cartesian product of a complete bipartite graph $K_{m,n}$ and $G$ in terms of $\qec (G)$. We first recall the definition of the Cartesian product of two graphs.

\begin{defn}\label{cproduct}
Let $G_1=(V_1, E_1)$ and $G_2=(V_2, E_2)$ be two connected graphs. The Cartesian of graphs $G_1$ and $G_2$, denoted by $G_1\times G_2$, is a graph with vertex set $V=V_1\times V_2$ and two vertices $(u_i,v_k),(u_j,v_l)\in V$ are adjacent in $G_1\times G_2$ if and only if either $u_i\sim u_j$ and $v_k=v_l$, or $u_i= u_j$ and $v_k\sim v_l$.
\end{defn}
\begin{ex}
Example of the Cartesian product of $P_3$ and $C_3$:
\begin{figure}[H]
\begin{center}
\begin{tikzpicture}[scale=1]
\draw  (1,0)-- (2,0);
\draw  (2,0)-- (3,0);
\begin{scriptsize}
\fill (1,0) circle (2.5pt);
\draw (1.0,-0.4) node {$u_1$};
\fill (2,0) circle (2.5pt);
\draw (2.,-0.4) node {$u_2$};
\fill (3,0) circle (2.5pt);
\draw (3.,-0.4) node {$u_3$};
\end{scriptsize}
\end{tikzpicture}
\hspace{1cm}
\begin{tikzpicture}[scale=1]
\draw  (1,1)-- (2,2);
\draw  (2,2)-- (3,1);
\draw  (3,1)-- (1,1);
\begin{scriptsize}
\fill (2.,2.) circle (2.5pt);
\draw (2.,2.35) node {$v_1$};
\fill (1.,1.) circle (2.5pt);
\draw (1.,0.65) node {$v_2$};
\fill (3.,1.) circle (2.5pt);
\draw (3.,0.65) node {$v_3$};
\end{scriptsize}
\end{tikzpicture}
\hspace{1cm}
\begin{tikzpicture}[scale=1]
\draw  (2.5,3.6)-- (3.62,2.98);
\draw  (2.5,3.6)-- (4.,2.);
\draw  (2.5,3.6)-- (4.,1.);
\draw  (3.62,2.98)-- (4.,2.);
\draw  (3.62,2.98)-- (3.,0.);
\draw  (4.,2.)-- (2.,0.);
\draw  (4.,1.)-- (3.,0.);
\draw  (4.,1.)-- (2.,0.);
\draw  (1.,1.)-- (4.,1.);
\draw  (3.,0.)-- (2.,0.);
\draw  (3.,0.)-- (1.,2.);
\draw  (2.,0.)-- (1.42,2.98);
\draw  (1.,1.)-- (1.,2.);
\draw  (1.42,2.98)-- (1.,1.);
\draw  (1.42,2.98)-- (1.,2.);
\begin{scriptsize}
\fill (2.,0.) circle (2.5pt);
\draw (1.85,-0.25) node {$(u_2,v_3)$};
\fill (3.,0.) circle (2.5pt);
\draw (3.15,-0.25) node {$(u_2,v_2)$};
\fill (4.,1.) circle (2.5pt);
\draw (4.8,1.) node {$(u_2,v_1)$};
\fill (1.,1.) circle (2.5pt);
\draw (0.2,1.) node {$(u_3,v_1)$};
\fill (1.,2.) circle (2.5pt);
\draw (0.2,2.) node {$(u_3,v_2)$};
\fill (4.,2.) circle (2.5pt);
\draw (4.8,2.) node {$(u_1,v_3)$};
\fill (1.42,2.98) circle (2.5pt);
\draw (0.6,3.) node {$(u_3,v_3)$};
\fill (3.62,2.98) circle (2.5pt);
\draw (4.4,3.) node {$(u_1,v_2)$};
\fill (2.5,3.6) circle (2.5pt);
\draw (2.52,3.97) node {$(u_1,v_1)$};
\end{scriptsize}
\end{tikzpicture}
\end{center}
\caption{$P_3$, $C_3$ and $P_3\times C_3$ (Left to right)}
\label{figure2}
\end{figure}
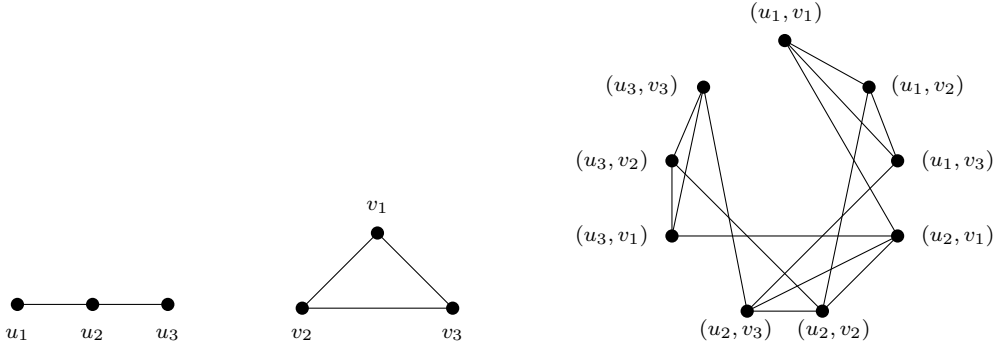
\end{ex}
Let $G_1$ and $G_2$ be two connected graphs with vertex set $V_1=\{u_1,u_2,\ldots ,u_m\}$ and $V_2=\{v_1,v_2,\ldots ,v_n\}$, respectively. Suppose that $M_k=\{(u_k,v_1),(u_k,v_2),\ldots ,(u_k,v_n)\}$ for all $k\in [m]$. Then the vertex set of $G_1\times G_2$ is $\cup _{k=1}^mM_k$. Then the distance matrix $D_{G_1\times G_2}$ is an $m\times m$ block matrix whose each block is of order $n$. Let $(u_i,v_p)$ and $(u_j,v_q)$ be two arbitrary elements from $M_i$ and $M_j$, respectively. Let $u_i\sim u_{i_1}\sim u_{i_2}\sim \cdots \sim u_{i_l}\sim u_j$ and $v_p\sim v_{p_1}\sim v_{p_2}\sim \cdots \sim v_{p_r}\sim v_q$ be two paths in $G_1$ and $G_2$, respectively. Then
\begin{center}
$(u_i,v_p)\sim (u_{i_1},v_p)\sim \cdots \sim (u_{i_l},v_p)\sim (u_j,v_p)\sim (u_j,v_{p_1})\sim \cdots \sim (u_j,v_{p_r})\sim (u_j,v_q)$
\end{center}
is a path from $(u_i,v_p)$ to $(u_j,v_q)$ in $G_1\times G_2$ and thus $G_1\times G_2$ is also a connected graph. Also note that any path between $(u_i,v_p)$ and $(u_j,v_q)$ in $G_1\times G_2$ gives us a path between $u_i$ and $u_j$ in $G_1$, and a path between $v_p$ and $v_q$ in $G_2$.  This implies
\begin{equation}\label{distancerelation}
 d_{G_1\times G_2}((u_i,v_p),(u_j,v_q))=d_{G_1}(u_i,u_j)+d_{G_2}(v_p,v_q).       
\end{equation}
Thus the $ij$-th block of $D_{G_1\times G_2}$ formed by rows $M_i$ and columns $M_j$ is $d_{G_1}(u_i,u_j)\ J+D_{G_2}$, and
\[
    D_{G_1\times G_2}=\begin{pmatrix}
D_{G_2} & D_{G_2}+d_{G_1}(u_1,u_2)J & \hdots & D_{G_2}+d_{G_1}(u_1,u_m)J \\
D_{G_2}+d_{G_1}(u_2,u_1)J & D_{G_2} & \hdots & D_{G_2}+d_{G_1}(u_2,u_m)J \\
\vdots & \vdots & \ddots & \vdots \\
D_{G_2}+d_{G_1}(u_m,u_1)J & D_{G_2}+d_{G_1}(u_m,u_2)J &\hdots  & D_{G_2}\\
\end{pmatrix}.
\]
We now recall an interesting result that gives us the  quadratic embedding constant of the Cartesian product of two arbitrary graphs of QE class. 
\begin{thm}\cite[Theorem 3.3]{zakiyyah}\label{0qec}
	Let $G_1$ and $G_2$ be two non-trivial graphs of QE class. Then, $\qec (G_1\times G_2)=0$.
\end{thm}

We now prove Proposition \ref{qecrelation} -- we derive a lower bound of $\qec (G_1\times G_2)$ in terms of $\qec (G_1)$ and $\qec (G_2)$.
\begin{proof}[Proof of Proposition~\ref{qecrelation}]
	Let $x\in \mathbb{R}^n$ such that $\langle e,x\rangle =0$, $\langle x,x\rangle =1$ and $\qec (G_2)=x^TD_{G_2}x$. Define $y\in \mathbb{R}^{mn}$ such that $y:=\frac{1}{\sqrt{m}}(\underbrace{x^T,x^T,\ldots ,x^T}_{m\ times})^T$. Then $\langle e,y\rangle =0$ and $\langle y,y\rangle =1$. Since $Jx=0$, we have
	\begin{equation*}
		y^TD_{G_1\times G_2}y=\frac{1}{m}(x^T,x^T,\ldots ,x^T)\begin{pmatrix}
			mD_{G_2}x \\
			mD_{G_2}x \\
			\vdots \\
			mD_{G_2}x
		\end{pmatrix}=mx^TD_{G_2}x=m\qec (G_2).
	\end{equation*}
	Thus $m\qec (G_2)\leq \qec (G_1\times G_2)$. Also, by interchanging the role of $G_1$ and $G_2$, we can show that $n\qec (G_1)\leq \qec (G_1\times G_2)$. Thus, $\max \{n\qec (G_1),m \qec (G_2)\}\leq\qec (G_1\times G_2)$.
\end{proof}

\begin{rem}
	Note that the equality in \eqref{extremalinequality} does not always hold. For that, take $G_1=K_2$ and $G_2=K_{1,2}$. Then $\max \{3(-1),2(-\frac{2}{3})=-\frac{4}{3}$. But $\qec (K_2\times K_{1,2})$ is equal to zero by Theorem \ref{0qec}.
\end{rem}


As an immediate consequence of Proposition \ref{qecrelation}, we conclude the following.
\begin{cor}\label{qeccatprodiff}
Let $G_1$ and $G_2$ be two connected graphs. Then $G_1\times G_2$ is of QE class if and only if both $G_1$ and $G_2$ are of QE class.
\end{cor}
\begin{proof}
	If $G_1$ and $G_2$ are of QE class, then by Theorem \ref{0qec}, $G_1\times G_2$ is of QE class.  To prove the converse, let either $G_1$ or $G_2$ be of the non-QE class. By Proposition \ref{qecrelation}, $0<\qec(G_1\times G_2)$, a contradiction since $G_1\times G_2$ is a graph of the QE class. Thus $G_1$ and $G_2$ are of QE class.
\end{proof}
We next study the extremal case of the inequality \eqref{extremalinequality} by giving a explicit formula for the quadratic embedding constant of the Cartesian product graphs $K_m\times G$ and $K_{m,n}\times G$, where $G$ is an arbitrary graph. For a graph $G$ of QE class, by Theorem \ref{0qec},  $\qec (K_m\times G)=0$. We now obtain a formula of $\qec(K_m\times G)$ for an arbitrary connected graph $G$ of non-QE class.
\begin{proof}[Proof of Theorem~\ref{theorem4}]
Let $V$ be the vertex set of $K_m\times G$. Note that
\begin{equation*}\label{distancematrix}
    D_{K_m\times G}=\begin{pmatrix}
D_G & D_G+J & \hdots & D_G+J \\
D_G+J & D_G & \hdots & D_G+J \\
\vdots & \vdots & \ddots & \vdots \\
D_G+J & D_G+J &\hdots  & D_G\\
\end{pmatrix}.
\end{equation*}
Let $\mathcal{S}(D_{K_m\times G})$ be the set of all $(f,\lambda ,\mu)\in (C(V)\cong \mathbb{R}^{mn})\times \mathbb{R} \times \mathbb{R}$ satisfying
\begin{align}
(D_{K_m\times G}-\lambda I)f &=\frac{\mu}{2}\  e \label{mainequation},\\
\langle f,f \rangle &=1 \label{one},\\
\langle e,f \rangle &=0. \label{zero}
\end{align}
By Proposition \ref{qecformula}, $\qec(K_m\times G)=\max \{\lambda :(f,\lambda ,\mu)\in \mathcal{S}(D_{K_m\times G})\}$. Suppose $f=({x^1}^T,{x^2}^T,\ldots ,{x^m}^T)^T$, where $x^i\in \mathbb{R}^n$ for all $i\in [m]$. Then, the equation \eqref{mainequation} gives us a system of $m$ equations. Subtracting the $(i+1)$-th equation from the $i$-th equation, we get
\begin{equation}\label{jequation}
(J+\lambda I)(x^i-x^{i+1})=0, \ \ \forall \ \ i\in [m-1].
\end{equation}
Since $G$ is a connected graph of non-QE class, by Proposition \ref{qecrelation}, it is enough to consider $0<\lambda $. From \eqref{jequation}, we have $x^i=x^{i+1}$ for all $i\in [m-1]$ and equation \eqref{zero} gives us $\langle e,x^1\rangle =0$. Thus the equations \eqref{mainequation}-\eqref{zero} reduces to
\begin{align}
(mD_G-\lambda I)x^1 &=\frac{\mu}{2}\  e \label{mainequation2},\\
m\langle x^1,x^1 \rangle &=1 \label{one2},\\
\langle e,x^1 \rangle &=0. \label{zero2}
\end{align}
Notice that $\qec (G)=\max \{a: (x,a,\eta)\in \mathcal{S}(D_G)\}$, where $\mathcal{S}(D_G)$ is the set of all stationary points $(x,a,\eta)\in \mathbb{R}^n\times \mathbb{R} \times \mathbb{R}$ satisfying
\begin{equation}\label{aequation}
(D_G-aI)x=\frac{\eta}{2}\  e,\ \ \  \langle x,x \rangle =1,\ \ \  \langle e,x \rangle =0.
\end{equation}
From equations \eqref{mainequation2}-\eqref{aequation}, the following relations
\begin{equation*}
\lambda =ma,\ \ \  x^1=\frac{1}{\sqrt{m}}\ x,\ \ \  \mu =\sqrt{m}\  \eta,
\end{equation*}
gives us a one to one correspondence between $\mathcal{S}(D_{K_m\times G})$ and $\mathcal{S}(D_G)$. Thus
\begin{align*}
 \qec (K_m\times G) &=\max \{\lambda :(({x^1}^T,{x^2}^T,\ldots ,{x^m}^T)^T,\lambda ,\mu)\in \mathcal{S}(D_{K_m\times G})\}\\
 &=\max \{ma :(x,a,\eta)\in \mathcal{S}(D_G)\}\\
 &=m\qec (G).\end{align*}\end{proof}
We next derive the quadratic embedding constant of the Cartesian product of $K_{m,n}$ and an arbitrary connected graph $G$. To proceed, we require a preliminary result for the quadratic embedding constant of a complete bipartite graph $K_{m,n}$.
\begin{thm}\label{qeccbip}\cite[Theorem 2.8]{zakiyyah}
	Let $K_{m,n}$ be a complete bipartite graph on $(m+n)$ vertices. Then
	\begin{center}
		$\qec(K_{m,n})=\frac{m(n-2)+n(m-2)}{m+n}$\ \ \  for all \ \ \  $m\geq 1$, $n\geq 1$.
	\end{center}
    Moreover, $K_{m,n}$ is of QE class if and only if one of the following holds:
    \begin{enumerate}
        \item [\rm{(i)}] Either $m=1$ or $n=1$.
        \item [\rm{(ii)}] $m=n=2$.
    \end{enumerate}
\end{thm}
Let $G$ be a connected graph on $l$ vertices and $K_{m,n}$ be a complete bipartite graph on $m+n$ vertices. Without loss of generality, assume that $m\geq n$. For $m=n=1$, $K_{1,1}=K_2$ and by Theorem \ref{theorem4}, $\qec(K_{1,1}\times G)=2\qec(G)$ for any connected graph $G$ of non-QE class. Let $m>1$. Then, all possibilities of $m$ and $n$ can be expressed in three disjoint sets: $(a)\ m> 1, n=1$, $(b)\ m=n=2$ and $(c)\ m>2, n\geq 2$. Note that the distance matrix of $K_{m,n}\times G$ is
\[D_{K_{m,n}\times G}=\begin{pmatrix}
D_G & D_G+2J & \hdots & D_G+2J & D_G+J & D_G+J & \hdots & D_G+J\\
D_G+2J & D_G & \hdots & D_G+2J & D_G+J & D_G+J & \hdots & D_G+J \\
\vdots & \vdots & \ddots & \vdots & \vdots & \vdots & \hdots & \vdots \\
D_G+2J & D_G+2J &\hdots  & D_G & D_G+J & D_G+J & \hdots & D_G+J\\
D_G+J & D_G+J &\hdots  & D_G+J & D_G & D_G+2J & \hdots & D_G+2J\\
D_G+J & D_G+J &\hdots  & D_G+J & D_G+2J & D_G & \hdots & D_G+2J\\
\vdots & \vdots & \hdots & \vdots & \vdots & \vdots & \ddots & \vdots \\
D_G+J & D_G+J &\hdots  & D_G+J & D_G+2J & D_G+2J & \hdots & D_G\\
\end{pmatrix}.
\]
Now, $\qec(K_{m,n}\times G)=\max \{\lambda :(f,\lambda ,\mu)\in \mathcal{S}(D_{K_{m,n}\times G})\}$, where $\mathcal{S}(D_{K_{m,n}\times G})$ is the set of all stationary points $(f,\lambda ,\mu)\in \mathbb{R}^{(m+n)l}\times \mathbb{R} \times \mathbb{R}$ satisfying
\begin{align}
(D_{K_{m,n}\times G}-\lambda I)f &=\frac{\mu}{2}\  e \label{cbmainequation},\\
\langle f,f \rangle &=1 \label{cbone},\\
\langle e,f \rangle &=0. \label{cbzero}
\end{align}
Notice that the equation \eqref{cbmainequation} gives us a system of $m+n$ equations. Suppose that $f=({x^1}^T,{x^2}^T,\ldots ,{x^{m+n}}^T)^T$, where $x^i\in \mathbb{R}^l$ for all $i\in [m+n]$. Then, subtracting the $(i+1)$-th equation from the $i$-th equation, we get the following relations for different cases.\\
{\bf Case I.} For $m=n=2$ and $m>2, n\geq 2$,
\begin{equation}\label{cbjequation1}
(J+\frac{\lambda}{2} I)(x^i-x^{i+1})=0, \ \forall \ 1\leq i\leq m-1 \ \&\  \ m+1\leq i\leq m+n-1
\end{equation}
and subtracting the $(m+1)$-th equation from the $m$-th equation, we have
\begin{equation}\label{cbjequation2}
\sum _{p=1}^{m-1}Jx^p -Jx^m+Jx^{m+1}-\sum _{q=m+2}^{m+n}Jx^q +\lambda (x^{m+1}-x^m)=0.
\end{equation}
{\bf Case II.} For $m> 1, n=1$,
\begin{equation}\label{secondcbjequation1}
(J+\frac{\lambda}{2} I)(x^i-x^{i+1})=0, \ \forall \ 1\leq i\leq m-1
\end{equation}
and subtracting the $(m+1)$-th equation from the $m$-th equation, we have
\begin{equation}\label{secondcbjequation2}
\sum _{p=1}^{m-1}Jx^p -Jx^m+Jx^{m+1}+\lambda (x^{m+1}-x^m)=0.
\end{equation}
We are interested in the Cartesian product $K_{m,n}\times G$ where either $K_{m,n}$ or $G$ is of non-QE class. Otherwise, the quadratic embedding constant is zero by Theorem \ref{0qec}. Thus, it suffices to consider $\lambda >0$, and from \eqref{cbjequation1}, \eqref{secondcbjequation1},
\begin{equation}\label{lambda1}
x^i=x^{i+1}, \forall \ 1\leq i\leq m-1 \ \&\  \ m+1\leq i\leq m+n-1
\end{equation}
and
\begin{equation}\label{lambda2}
x^i=x^{i+1}, \forall \ 1\leq i\leq m-1,
\end{equation}
respectively. Also, recall that $\qec (G)=\max \{a: (x,a,\eta)\in \mathcal{S}(D_G)\}$, where $\mathcal{S}(D_G)$ is the set of all stationary points $(x,a,\eta)\in \mathbb{R}^l\times \mathbb{R} \times \mathbb{R}$ satisfying
\begin{equation}\label{cbaequation}
(D_G-aI)x=\frac{\eta}{2}\  e,\ \ \  \langle x,x \rangle =1,\ \ \  \langle e,x \rangle =0.
\end{equation}
Keeping these relations in hand, we now prove Theorem \ref{theorem5}.
\begin{proof}[Proof of Theorem~\ref{theorem5} (i)]
Without loss of generality assume that $m\geq n$. By Theorem \ref{qeccbip}, $K_{m,n}$ is of QE class if and only if either $m=n=2$ or $m\geq 1,n=1$. Thus, we consider two cases.\\
{\bf Case I.} $m=n=2$:  From \eqref{cbjequation2} and \eqref{lambda1}, $x^1=x^2=x^3=x^4$. This reduces the equations \eqref{cbmainequation}-\eqref{cbzero} as follows.
\begin{equation}\label{modifiedequation}
(4D_G-\lambda I)x^1 =\frac{\mu}{2}\  e,\ \ \  4\langle x^1,x^1 \rangle =1,\ \ \  \langle e,x^1 \rangle =0.
\end{equation}
Using \eqref{cbaequation} and \eqref{modifiedequation}, we have
\begin{align*}
 \qec (K_{2,2}\times G) &=\max \{\lambda :(({x^1}^T,{x^2}^T,{x^3}^T,{x^{4}}^T)^T,\lambda ,\mu)\in \mathcal{S}(D_{K_{2,2}\times G})\}\\
 &=\max \{4a :(x,a,\eta)\in \mathcal{S}(D_G)\}\\
 &=4\qec (G).
\end{align*}
{\bf Case II.} $m\geq 1,n=1$: For $m=n=1$, by Theorem \ref{theorem4}, $\qec(K_{1,1}\times G)=2\qec(G)$. Suppose $m>1,n=1$. By equation \eqref{lambda2}, $x^i=x^{i+1}$ for all $i\in [m-1]$ and equation \eqref{secondcbjequation2} implies that
\begin{equation}\label{mjequation}
    (m-2)Jx^m+Jx^{m+1}+\lambda (x^{m+1}-x^m)=0.
\end{equation}
Also, using \eqref{cbzero}, we get
\begin{equation}\label{modifiedjequation}
    mJx^m+Jx^{m+1}=0.
\end{equation}
Multiplying equations \eqref{mjequation} and \eqref{modifiedjequation} by $m+1$ and $(m-2)+1$ respectively, and  then computing their difference, we get
\begin{equation*}
    \{J+\frac{\lambda }{2} (m+1)I\}(x^m-x^{m+1})=0.
\end{equation*}
Since $\lambda >0$, $x^m=x^{m+1}$. So, equations \eqref{cbmainequation}-\eqref{cbzero} reduce to
\begin{equation*}\label{modifiedequation2}
\{(m+1)D_G-\lambda I\}x^1 =\frac{\mu}{2}\  e,\ \ \  (m+1)\langle x^1,x^1 \rangle =1,\ \ \  \langle e,x^1 \rangle =0,
\end{equation*}
and using \eqref{cbaequation} we have
\begin{align*}
 \qec (K_{m,1}\times G) &=\max \{\lambda :(({x^1}^T,{x^2}^T,\ldots ,{x^{m+1}}^T)^T,\lambda ,\mu)\in \mathcal{S}(D_{K_{m,1}\times G})\}\\
 &=\max \{(m+1)a :(x,a,\eta)\in \mathcal{S}(D_G)\}\\
 &=(m+1)\qec (G).
\end{align*}
\end{proof}
\begin{proof}[Proof of Theorem~\ref{theorem5} (ii)] By \eqref{lambda1}, $x^i=x^{i+1}$ for all $1\leq i\leq m-1 \ \&\  \ m+1\leq i\leq m+n-1$ and from \eqref{cbjequation2}, we have
\begin{equation}\label{mnjequation}
    (m-2)Jx^m-(n-2)Jx^{m+1}+\lambda (x^{m+1}-x^m)=0.
\end{equation}
From \eqref{cbzero}, we have
\begin{equation}\label{modifiedjequation2}
    mJx^m+nJx^{m+1}=0.
\end{equation}
Now, multiplying \eqref{mnjequation} and \eqref{modifiedjequation2} by $m+n$ and $(n-2)-(m-2)$ respectively and adding them, we  get
\begin{equation}\label{mjequation5}
    \{J-\lambda \frac{m+n}{m(n-2)+n(m-2)}I\}(x^m-x^{m+1})=0.
\end{equation}
We now divide the remaining proof into two cases.\\
{\bf Case I.} Suppose $\lambda \neq l\frac{n(m-2)+m(n-2)}{m+n}$. Then, from \eqref{mjequation5}, $x^m=x^{m+1}$. Again, equations \eqref{cbmainequation}-\eqref{cbzero} reduce to
\begin{equation*}\label{modifiedequation3}
\{(m+n)D_G-\lambda I\}x^1 =\frac{\mu}{2}\  e,\ \ \  (m+n)\langle x^1,x^1 \rangle =1,\ \ \  \langle e,x^1 \rangle =0,
\end{equation*}
and from \eqref{cbaequation}, we have
\begin{align*}
 \qec (K_{m,n}\times G) &=\max \{\lambda :(({x^1}^T,{x^2}^T,\ldots ,{x^{m+n}}^T)^T,\lambda ,\mu)\in \mathcal{S}(D_{K_{m,n}\times G})\}\\
 &=\max \{(m+n)a :(x,a,\eta)\in \mathcal{S}(D_G)\}\\
 &=(m+n)\qec (G).
\end{align*}
{\bf Case II.} Suppose $\lambda = l\frac{n(m-2)+m(n-2)}{m+n}$. Take
\[x^1=x^2=\cdots =x^m=-\frac{\sqrt{n}}{\sqrt{ml(m+n)}}\ e,~~
x^{m+1}=x^{m+2}=\cdots =x^{m+n}=\frac{\sqrt{m}}{\sqrt{nl(m+n)}}\ e,
\]
and
\begin{equation*}
\mu =\Big (\frac{n-m}{m+n} \Big )\sqrt{\frac{mnl}{m+n}}.
\end{equation*}
By a simple calculation, one can see that the equations \eqref{cbmainequation}-\eqref{cbzero} are satisfied by the above choice of $(({x^1}^T,{x^2}^T,\ldots ,{x^{m+n}}^T)^T,\lambda ,\mu)\in \mathbb{R}^{(m+n)l}\times \mathbb{R}\times \mathbb{R}$. Hence,
\begin{center}
$\qec (K_{m,n}\times G)=\max \{(m+n)\qec (G),l\frac{n(m-2)+m(n-2)}{m+n}\}$.
\end{center}
\end{proof}
	We conclude this section by presenting a few consequences of Theorem \ref{theorem5}. 
\begin{cor}
Let $G=K_{p_1,p_2,\ldots ,p_k}$ and $p_1\geq p_2\geq \ldots \geq p_k$. Then, from Theorem \ref{theorem5} and \cite[Theorem 1.2]{bata}, we can conclude that for all $ m\geq 3,n\geq 2$,
\begin{align*}
 &\qec (K_{m,n}\times K_{p_1,p_2,\ldots ,p_k})\\
 &=\begin{cases}
\max \{(m+n)\frac{p_2(p_1-2)+p_1(p_2-2)}{p_1+p_2},(p_1+p_2)\frac{n(m-2)+m(n-2)}{m+n}\}, & \mbox{if }k=2, p_1\geq 3, p_2\geq 2, \\
\max \{(m+n)\frac{(p_1-2)(k-2)-2}{p_1+k-1},(p_1+k-1)\frac{n(m-2)+m(n-2)}{m+n}\}, & \mbox{if }\begin{cases}
   \{k\geq 3, p_1\geq 5, p_2=p_3=\cdots =p_k\} \mbox{ or,} \\
   \{k\geq 4, p_1= 4, p_2=p_3=\cdots =p_k\} \mbox{ or,} \\
   \{k\geq 5, p_1= 3, p_2=p_3=\cdots =p_k\},
\end{cases}, \\
(p_1+p_2+\ldots +p_k)\frac{n(m-2)+m(n-2)}{m+n} , & \mbox{otherwise}.
\end{cases}
\end{align*}
\end{cor}
\begin{ex}
Let $G$ be the $p$-fold edge amalgamation of $K_3$, that is, the tripartite graph $K_{p,1,1}$. Then $\qec (K_{p,1,1})=\frac{p-4}{p+2}\ \forall \ p\geq 1$ \cite[Theorem 2.9]{zakiyyah}. Thus, for all $ m\geq 3,n\geq 2$,
\begin{align*}
 \qec (K_{m,n}\times K_{p,1,1}) &=\max \{(m+n)\frac{p-4}{p+2},(p+2)\frac{n(m-2)+m(n-2)}{m+n}\}\\
 &=(p+2)\frac{n(m-2)+m(n-2)}{m+n}.
\end{align*}
\end{ex}
\section*{Acknowledgments}
P.N. Choudhury was partially supported by Prime Minister Early Career Research Grant (PMECRG) ANRF/ECRG/2024/002674/PMS (ANRF, Govt. of India), INSPIRE Faculty Fellowship research grant DST/INSPIRE/04/2021/002620 (DST, Govt. of India), and IIT Gandhinagar Internal Project: IP/IP/50025. R. Nandi was partially supported by IIT Gandhinagar Post-Doctoral Fellowship IP/IP/50020, Institute Seed Research Grant (ISRG), Indian Institute of Information Technology Design and Manufacturing, Kurnool.

\end{document}